\documentclass[reqno,12pt]{amsart}

\usepackage{tikz}
\usepackage{amsmath}
\usepackage{amsfonts}
\usepackage{amssymb,euscript, mathrsfs,amscd} 
\usepackage{epsfig}
\usepackage{a4wide}
\usepackage{blkarray}
\usepackage{multirow}
\usepackage{mathtools}
\usepackage{bm}
\usepackage{ytableau}
\usepackage{caption}
\usepackage{hyperref}
\usepackage{cases}
\usepackage[shortlabels]{enumitem}

\usepackage{graphicx,subfigure}

\usetikzlibrary{decorations.markings}

\newtheorem{thm}{Theorem}[section]
\newtheorem{prop}[thm]{Proposition}
\newtheorem{lem}[thm]{Lemma}
\newtheorem{cor}[thm]{Corollary}
\newtheorem{conj}[thm]{Conjecture}
\theoremstyle{definition}
\newtheorem{defn}[thm]{Definition}
\newtheorem{exam}[thm]{Example}
\newtheorem{rmk}[thm]{Remark}
\theoremstyle{remark}
\newtheorem*{notation}{Notations}

\newcommand\asc{\operatorname{asc}}
\newcommand\des{\operatorname{des}}
\newcommand\inv{\operatorname{inv}}
\newcommand\inc{\operatorname{inc}}
\newcommand\sgn{\operatorname{sgn}}
\newcommand\sh{\operatorname{sh}}
\newcommand\rank{\operatorname{rank}}

\newcommand\wt{\operatorname{wt}}

\newcommand\WM{\operatorname{WM}}

\newcommand\Des{\operatorname{Des}}

\newcommand\Par{\operatorname{Par}}

\newcommand\Sym{\mathsf{Sym}}
\newcommand\QSym{\mathsf{QSym}}

\newcommand\trigeq{\unrhd}

\newcommand\TT{\mathcal{T}}
\newcommand\NN{\mathcal{N}}
\newcommand\UU{\mathcal{U}}
\newcommand\EE{\mathcal{E}}
\newcommand\EP{\mathcal{EP}}
\newcommand\SSS{\mathcal{S}}
\newcommand\II{\mathcal{I}}
\newcommand\LL{\mathcal{L}}
\newcommand\HH{\mathcal{H}}

\newcommand\xx{\mathbf{x}}
\newcommand\yy{\mathbf{y}}
\newcommand\uu{\mathbf{u}}
\newcommand\mm{\mathbf{m}}

\newcommand\ee{\mathfrak{e}}
\newcommand\hh{\mathfrak{h}}
\newcommand\sch{\mathfrak{J}}
\newcommand\pp{\mathfrak{p}}
\newcommand\mono{\mathfrak{m}}
\newcommand\ff{\mathfrak{f}}

\newcommand\ww{\mathsf{w}}
\newcommand\vv{\mathsf{v}}

\newcommand\pathp{\mathsf{p}}

\newcommand\oo{\mathfrak{o}}

\newcommand\RECTX{2.5}
\newcommand\RECTY{1}

\newcommand\equivclass[1]{#1/{\sim}}
\newcommand\qand{\quad\mbox{and}\quad}

\title[Noncommutative \( P \)-symmetric functions]{Chromatic quasisymmetric functions and noncommutative $P$-symmetric functions}

\author{Byung-Hak Hwang}
\email{byunghakhwang@gmail.com}


\begin{document}
\begin{abstract}
  For a natural unit interval order $P$, we describe proper colorings of
  the incomparability graph of $P$ in the language of heaps.
  We also introduce a combinatorial operation, called a \emph{local flip},
  on the heaps. This operation defines an equivalence relation on the proper
  colorings, and the equivalence relation refines the ascent statistic introduced
  by Shareshian and Wachs.

  In addition, we define an analogue of noncommutative symmetric functions
  introduced by Fomin and Greene, with respect to $P$.
  We establish a duality between
  the chromatic quasisymmetric function of $P$ and these noncommutative
  symmetric functions. This duality leads us to positive expansions of the
  chromatic quasisymmetric functions into several symmetric function bases.
  In particular, we present some partial results for the $e$-positivity conjecture.
\end{abstract}

\maketitle
\tableofcontents

\section{Introduction} \label{sec:intro}
\subsection{Chromatic quasisymmetric functions}
Chromatic quasisymmetric functions are one of the most notable objects in
algebraic combinatorics, because of their connections with other fields.
In \cite{Stanley1995}, Stanley introduced \emph{chromatic symmetric functions}
which generalize chromatic polynomials. For a graph $G$,
the chromatic symmetric function $X_G(\xx)$ of $G$ is defined by
\[
  X_G(\xx) = \sum_{\kappa} x^\kappa,
\]
where $\xx = \{x_1, x_2, \dots\}$ is a set of formal commuting variables,
$\kappa$ ranges over all proper colorings of $G$ and
$x^\kappa=\prod_{v\in V(G)} x_{\kappa(v)}$.
By definition, $X_G(1^n)$ counts proper colorings of $G$ with $n$ colors.
Stanley presented many interesting properties of $X_G(\xx)$ and a famous
conjecture; the so-called \emph{$e$-positivity conjecture}
(Conjecture~\ref{conj:original_e_positivity_conj}).
In fact, the $e$-positivity conjecture is equivalent to a conjecture on
immanants suggested by Stanley and Stembridge~\cite{SS1993}.
This conjecture is one of the most famous long-standing open problems in
algebraic combinatorics.

Although chromatic symmetric functions are of some interest in their own right,
the quasisymmetric refinement of $X_G(\xx)$, due to Shareshian and
Wachs~\cite{SW2016}, provides connections with other mathematical objects.
The \emph{chromatic quasisymmetric function} $X_G(\xx,q)$ of a graph $G$
whose vertex set is $[n]=\{1,2,\dots,n\}$ is defined as follows:
\[
  X_G(\xx,q) = \sum_{\kappa} q^{\asc(\kappa)} x^\kappa,
\]
where $\kappa:[n]\rightarrow \mathbb{P}$ ranges over all proper colorings of $G$,
and $\asc(\kappa)$ is the number of edges $\{i,j\}$ such that $i<j$ and
$\kappa(i)<\kappa(j)$. By definition, $X_G(\xx,q)$ is in general a quasisymmetric
function, not a symmetric function. 
When \( G \) is the incomparability graph of a natural unit interval order,
Shareshian and Wachs showed that  the quasisymmetric function \( X_G(\xx,q) \)
turns out to be symmetric~\cite[Theorem~4.5]{SW2016}.
They also refined the $e$-positivity conjecture
(some notions will be defined in Section~\ref{sec:background}):
for a natural unit interval order $P$, $X_{\inc(P)}(\xx,q)$ is $e$-positive.
Note that Stanley's original conjecture~\cite{Stanley1995} is for more general
graphs, the incomparability graphs of \( \mathbf{(3+1)} \)-free posets,
but due to Guay-Paquet~\cite{Guay-Paquet2013}, it suffices to consider the
incomparability graphs of natural unit interval orders.

Since this quasisymmetric generalization was introduced, the chromatic
quasisymmetric functions have been more actively studied than ever before.
For example, Carlsson and Mellit~\cite{CM2018} gave a plethystic relation between
chromatic quasisymmetric functions and unicellular LLT polynomials.
Clearman, Hyatt, Shelton and Skandera~\cite{CHSS2016} observed a relationship between chromatic quasisymmetric functions and characters of Hecke algebras of type A. Their result refines the previous work of Haiman~\cite{Haiman1993}.

Also, there is an algebraic geometry perspective for understanding the chromatic quasisymmetric functions. Let $P$ be a natural unit interval order on $[n]$.
Tymoczko~\cite{Tymoczko2008a,Tymoczko2008b} defined an $\mathfrak{S}_n$-action on the (equivariant)
cohomology of the regular semisimple Hessenberg variety corresponding to $P$.
Based on Tymoczko's work, Shareshian and Wachs~\cite{SW2016} conjectured that
the Frobenius characteristic of the cohomology of the variety is equal to
$\omega X_P(\xx,q)$.
This conjecture was proved independently by Brosnan--Chow~\cite{BC2018} and
Guay-Paquet~\cite{Guay-Paquet2016}, and hence it allows us to use geometric
approaches for studying chromatic quasisymmetric functions. For example,
Harada and Precup conjectured a recurrence relation on the \( e \)-coefficients of
the chromatic quasisymmetric functions of natural
unit interval orders \( P \) (\cite[Conjecture~8.1]{HP2019}),
and proved it for a very special class of \( P \) via geometric techniques.
Consequently, they showed that the chromatic quasisymmetric functions of such
\( P \) is \( e \)-positive using the recurrence relation.

The Shareshian--Wachs quasisymmetric refinement provides a clue towards resolving the $e$-positivity conjecture.
The chromatic quasisymmetric function $X_G(\xx,q)$ contains more information about
colorings compared to $X_G(\xx)$, and hence it is obvious that the conjecture of
Shareshian--Wachs implies the original $e$-positivity conjecture of
Stanley--Stembridge.
However, thanks to the quasisymmetric generalization,
we can cluster colorings of $G$ according to their ascent statistics.
This clustering directs our attention to specific colorings rather than the entire set.
In this context, the quasisymmetric generalization gives us a hint for the
\( e \)-positivity conjecture.
This prompts the question of finding a natural refinement of the Shareshian--Wachs refinement.

One of the goals of the paper is to answer this question. We will introduce a
combinatorial operation, called a \emph{local flip}, on proper colorings.
This operation establishes an equivalence relation on proper colorings,
and all colorings belonging to each equivalence class have the same ascent statistic.
Consequently, this relation refines the ascent statistic, and hence, the Shareshian and Wachs quasisymmetric refinement.
We anticipate that our refinement will contribute to resolving the $e$-positivity
conjecture in a similar vein.

\subsection{Positivity of symmetric functions}
Whenever a new class of symmetric functions is introduced, a natural question
emerges regarding their positivity with respect to various symmetric function
bases.
Establishing the positivity of a given symmetric function entails
the development of numerous combinatorial and algebraic tools.
For example, approaches using RSK-like algorithms or crystals are efficient for
showing Schur positivity of a given symmetric function.
(Recently, Kim and Pylyavskyy~\cite{KP2021} gave an RSK-like algorithm for
natural unit interval orders with some conditions.)

Another well-developed tool is the theory of noncommutative symmetric functions.
A prototypical example is the plactic monoid, introduced by Lascoux and
Sch\"{u}tzenberger~\cite{LS1981}, which is the quotient of a free monoid by
the Knuth relations.
The plactic monoid allows us to embed the ring of symmetric functions into a noncommutative ring. Hence it translates many questions from the ring of symmetric functions to the noncommutative ring.
Consequently, certain questions in the noncommutative ring may be more readily addressed than their counterparts in the commutative version.
Especially, using this noncommutative approach, Fomin and Greene~\cite{FG1998} introduced noncommutative Schur functions to prove Schur positivity for various symmetric functions.
In \cite{BF2017}, Blasiak and Fomin gave a more general algebraic framework of this approach. Their framework also involves Lam's~\cite{Lam2008} and a part of Assaf's~\cite{Assaf2015} works.

We devote the second half of this paper to introducing analogues of noncommutative
symmetric functions which reflect properties of a given natural unit interval
order $P$. In the same spirit of \cite{BF2017}, these noncommutative symmetric
functions provide expansions of the chromatic quasisymmetric function of $P$ in
terms of several bases. This approach unifies many results about positivity of
chromatic quasisymmetric functions, including results in
\cite{Athanasiadis2015, CH2019, HP2019, SW2016}.
In particular, we provide partial results for the $e$-positivity conjecture.

\subsection{Outline}
In Section~\ref{sec:background}, we briefly review some necessary background for natural unit interval orders, symmetric functions, and chromatic quasisymmetric functions.
In Section~\ref{sec:heap_local_flip}, we review the definition of heaps,
and describe a relation between proper colorings and heaps.
After that, we define a local flip, which is an operation on heaps.
This operation is central in this paper. Via local flips, we suggest a more
refined $e$-positivity conjecture (Conjecture~\ref{conj:refined_e-positivity}).
We also construct graphs on words according to a natural unit interval order,
similar to \cite{Assaf2015, BF2017}.
In Sections~\ref{sec:Noncomm}, we define
noncommutative $P$-symmetric functions which are analogues of noncommutative
symmetric functions. They reflect properties of \( P \) and local flips.
After we establish a duality between chromatic quasisymmetric
functions and noncommutative $P$-symmetric functions, we give combinatorial
descriptions of the coefficients in the expansions of $X_P(\xx,q)$ into
forgotten symmetric functions (Theorem~\ref{thm:f-positive}),
power sum symmetric functions (Theorem~\ref{thm:p-positive}),
and Schur functions (Theorem~\ref{thm:s-positive}).
Especially, in Section~\ref{sec:Noncomm_mono}, we provide a combinatorial
description of coefficients of $e_\lambda(\xx)$ in the expansions of
$X_P(\xx,q)$ where $\lambda$ is of two-column shape or hook shape.
In addition, we prove the conjecture of Harada--Precup (Theorem~\ref{thm:m=em}), that is, we show that the recurrence relation on the \( e \)-coefficients
hold for arbitrary natural unit interval orders.
We end with some open questions in Section~\ref{sec:future_direction}.

\section{Background} \label{sec:background}
In this section, we review background materials about natural unit interval orders, symmetric functions, and chromatic quasisymmetric functions. 

\begin{notation}
We write $\mathbb{Z}$ (respectively, $\mathbb{N}$ and $\mathbb{P}$) for the set of (respectively, nonnegative and positive) integers. For $n\in\mathbb{P}$, we write $[n] := \{1,2,\dots,n\}$.
\end{notation}

\subsection{Natural unit interval orders and their incomparability graphs} \label{subsec:NUIO}
Fix a positive integer $n$, and let $\mm=(m_1,\dots,m_n)$ be a weakly
increasing integer sequence satisfying $i\le m_i \le n$ for each $i$.
The \emph{natural unit interval order $P=P(\mm)$} corresponding to $\mm$
is a poset on $[n]$ whose ordering $<_P$ is given by $i<_P j$ if $m_i<j$.
The term ``natural unit interval order'' arises from the following unit interval
model: given a natural unit interval order $P$, one can assign a unit interval on
the real line to each $i\in[n]$ such that $i<_P j$ if and only if the unit interval
assigned to $i$ completely lies to the left of the unit interval assigned to $j$.
For example, the natural unit interval order $P=P(2,4,5,5,5)$ and
its unit interval model are presented in Figure~\ref{fig:NUIO_2455}.
\begin{figure}
\centering
\includegraphics[width=0.5\linewidth]{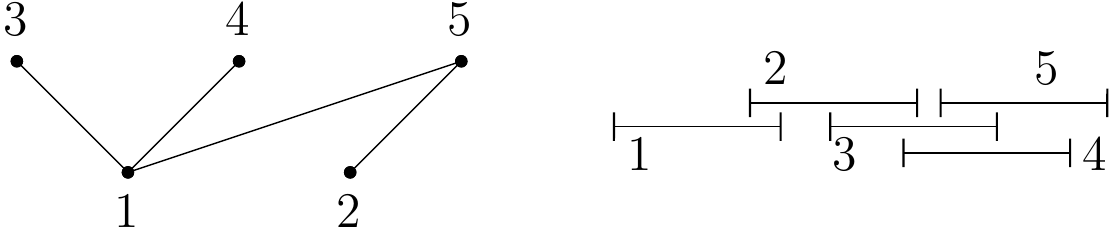}
\caption{The natural unit interval order $P(2,4,5,5,5)$ and its unit interval model.} \label{fig:NUIO_2455}
\end{figure}

For a poset $P$, the \emph{incomparability graph $\inc(P)$} of $P$ is the graph whose vertices are elements of $P$, and for $a,b\in P$, $a$ and $b$ are adjacent if and only if $a$ and $b$ are incomparable in $P$. When $P=P(\mm)$ is a natural unit interval order on $[n]$, for $i<j$, $i$ and $j$ are adjacent if and only if $j \le m_i$. In this paper, we deal with incomparability graphs of natural unit interval orders. By abuse of notation, we sometimes write $P$ instead of $\inc(P)$ for the incomparability graph of $P$.

An easy way to see the incomparability graph of $P(\mm)$ is to use Dyck diagrams.
A \emph{Dyck path} of length $2n$ is a lattice path from $(0,0)$ to $(n,n)$ using
$(1,0)$ and $(0,1)$ steps that never goes below the main diagonal.
There is a simple bijection between natural unit interval orders $P$ on $[n]$
and Dyck paths $D$ of length $2n$. Indeed, the sequence $\mm$ corresponding to $P$
records the heights of horizontal steps in $D$.
To see the incomparability graph of a natural unit interval order $P$,
draw the corresponding Dyck path.
Then unit square boxes which are placed between the diagonal line
and the Dyck path represent edges of $\inc(P)$; see Figure~\ref{fig:m=2455}.

Given a Dyck path $D$ of length $2n$, define the \emph{bounce path} of $D$ as
follows. Beginning at the origin $(0,0)$, the bounce path goes North until
it reaches an East step of $D$. Then it goes East until it reaches
the diagonal line $x=y$. This process continues recursively, and terminates
when it reaches $(n,n)$.
The \emph{height} $h(D)$ of $D$ is one less than the number of intersecting points of the bounce path and the diagonal line. Let $h(\mm) = h(D)$ where $D$ is the Dyck path corresponding to $P(\mm)$.
In fact, the height $h(\mm)$ is the maximum cardinality of independent sets of
$\inc(P(\mm))$, equivalently, the maximum length of chains in \( P(\mm) \).
See Figure~\ref{fig:m=2455} for an example.
\begin{figure}
\centering
\includegraphics[scale=0.6]{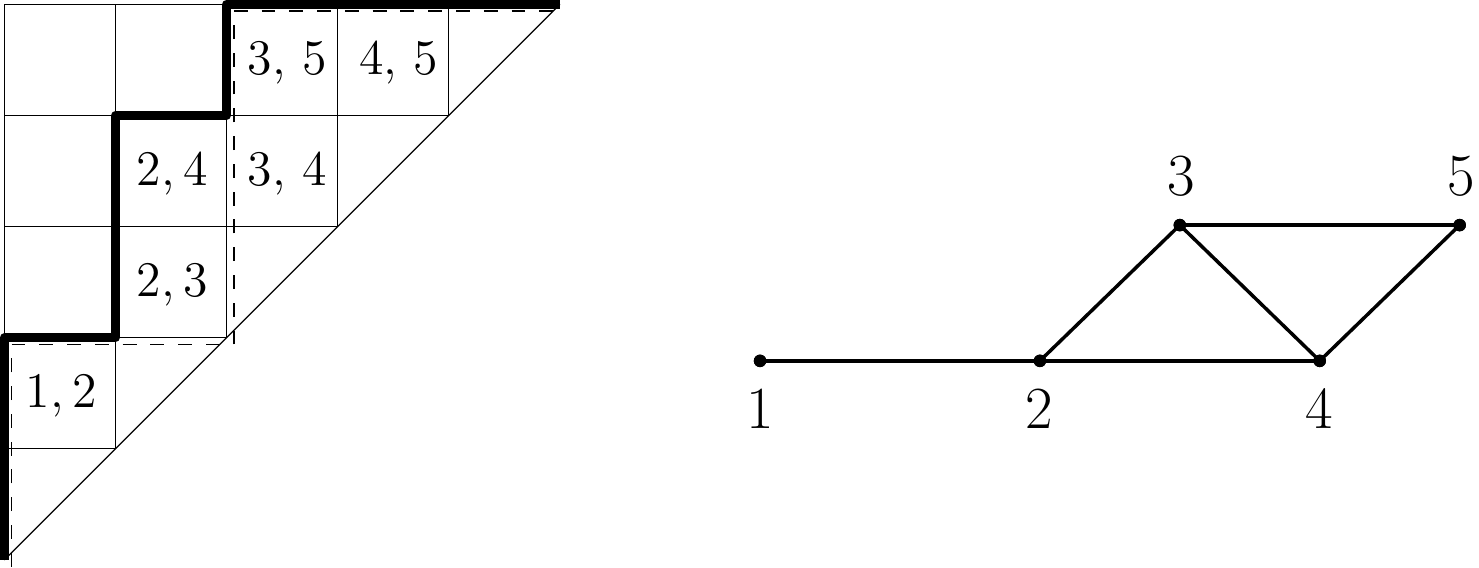}
\caption{The Dyck path corresponding to $\mm=(2,4,5,5,5)$ and the incomparability graph $\inc(P(\mm))$. The dashed line illustrates the bounce path, and the height $h(\mm) = 2$.} \label{fig:m=2455}
\end{figure}

For positive integers $a,b$, let $\mathbf{a+b}$ denote the poset which is a
disjoint union of two chains of length $a$ and $b$. For a poset $P$, we say that
$P$ is \emph{$\mathbf{(a+b)}$-free} if $P$ has no induced subposet isomorphic
to $\mathbf{a+b}$.
Introducing the chromatic symmetric functions, Stanley~\cite{Stanley1995} conjectured
that if $P$ is a $\mathbf{(3+1)}$-free poset, then $X_{\inc(P)}(\xx)$ is
$e$-positive (see Section~\ref{sec:symmetric_function}).
On the other hand, there is a well-known characterization~\cite{SS1958} such that
a poset $P$ is isomorphic to some natural unit interval order if and only if
$P$ is $\mathbf{(3+1)}$-free and $\mathbf{(2+2)}$-free. Thus the class of natural
unit interval orders is a subclass of the class of $\mathbf{(3+1)}$-free posets.
Guay-Paquet~\cite{Guay-Paquet2013}, however, showed that for a \( \mathbf{(3+1)} \)-free
poset \( P \), the chromatic symmetric function \( X_{\inc(P)}(\xx) \) can be
written as a positive linear combination of those of some natural unit interval
orders. We therefore restrict ourselves to natural unit interval orders throughout
this paper. See also Remark~\ref{rmk:2+2}.

\subsection{Words on the alphabet \texorpdfstring{$[n]$}{[n]}}
Fix a natural unit interval order $P$ on $[n]$. We now define some notions for
words on the alphabet $[n]$, with respect to $P$. For a word
$\ww=\ww_1\cdots\ww_d$,
we say that \( \ww \) is of \emph{type \( \mu=(\mu_1,\dots,\mu_n) \)} if the
letter \( i \) appears \( \mu_i \) times in \( \ww \) for each \( i \).
For $1\le i < d$, $i$ is a \emph{$P$-descent} of $\ww$ if $\ww_i >_P \ww_{i+1}$.
Define the \emph{$P$-descent set} $\Des_P(\ww)$ of $\ww$ by
\[
  \Des_P(\ww) := \{ i\in [d-1] \mid \ww_i >_P \ww_{i+1} \}.
\]
For $1\le i < j\le d$, $(i,j)$ is a \emph{\( P \)-inversion pair} of $\ww$
if $\ww_i$ and $\ww_j$ are incomparable in $P$, and $\ww_i>\ww_j$
in the natural order on $\mathbb{P}$.
We denote by $\inv_P(\ww)$ the number of \( P \)-inversion pairs of $\ww$.
For example, if $P=P(2,4,4,4)$ and $\ww=\mathsf{413231}$,
then $\Des_P(\ww) = \{1,5\}$ and $\inv_P(\ww) = 5$.

\subsection{Compositions and partitions}
A \emph{composition $\alpha$} of $n$, denoted by $\alpha\vDash n$, is a finite
sequence $(\alpha_1,\dots,\alpha_\ell)$ of positive integers whose sum equals $n$.
We call each integer $\alpha_i$ a \emph{part} of $\alpha$, and the \emph{length}
$\ell(\alpha)$ of $\alpha$ is the number of parts of $\alpha$.

There is a simple bijection between the set of compositions of \( n \) and the
collection of subsets of $[n-1]$:
given a composition $\alpha=(\alpha_1,\dots,\alpha_\ell)\vDash n$,
define $S(\alpha) := \{\alpha_1,\alpha_1+\alpha_2,\dots,\alpha_1+\cdots+\alpha_{\ell-1}\}\subset [n-1]$.
Conversely, for $S=\{s_1 < s_2 < \dots < s_k\}\subset [n-1]$,
define $\alpha(S) := (s_1, s_2-s_1, \dots, n-s_k)\vDash n$.
Then the two maps \( S \) and \( \alpha \) are inverse to each other.

A \emph{partition} $\lambda$ of $n$, denoted by $\lambda\vdash n$, is a
composition of $n$ satisfying $\lambda_1\ge \lambda_2 \ge \dots \ge \lambda_\ell$.
The \emph{conjugate} of \( \lambda \) is the partition \( \lambda'=(\lambda'_1,
\lambda'_2, \dots) \) where \( \lambda'_i \) is the number of parts of
\( \lambda \) greater than or equal to \( i \).
The \emph{Young diagram} of \( \lambda \) is a left-justified array of \( n \)
boxes with \( \lambda_i \) boxes in the \( i \)th row.
Note that the Young diagram of \( \lambda' \) is the transpose of the diagram
of \( \lambda \).
Denote by $\Par$ the set of partitions.

We will use two orderings on the set of partitions.
For partitions $\lambda,\mu\vdash n$,
we write $\lambda\trigeq\mu$ if $\lambda_1+\cdots+\lambda_i\ge \mu_1+\cdots+\mu_i$
for all $i\ge 1$. We call this order the \emph{dominance order}.
Another order we will use is the \emph{reverse lexicographic order}:
we say that \( \lambda \) is greater than or equal to \( \mu \) with respect to
the reverse lexicographic order if either $\lambda=\mu$, or for some $i$,
\[
\lambda_1=\mu_1,\quad\dots,\quad \lambda_i=\mu_i,\qand \lambda_{i+1}>\mu_{i+1}.
\]
The dominance order is a partial order while the reverse lexicographic order is a total order which is compatible with the dominance order.

\subsection{Symmetric and quasisymmetric functions} \label{sec:symmetric_function}
Let $\xx=(x_1, x_2, \dots)$ be a set of formal variables which commute with each other. A \emph{symmetric function $f(\xx)$} is a bounded degree formal power series in $\mathbb{Q}[[\xx]]$ which is invariant under permutation of the variables. For $n\ge 0$, let $\Sym_n$ be the vector space of symmetric functions of homogeneous degree $n$, and let $\Sym = \bigoplus_{n\ge 0} \Sym_n$ be the $\mathbb{Q}$-algebra of symmetric functions. It is well known that the dimension of $\Sym_n$ is the number of partitions of $n$. 

There are various bases for the space of symmetric functions $\Sym$:
\begin{enumerate}[label=(\roman*)]
\item The \emph{monomial symmetric function}: for a partition $\lambda\in\Par$,
\[
m_\lambda(\xx) := \sum_{\alpha} x^\alpha,
\]
where $\alpha$ ranges over all distinct permutations $\alpha=(\alpha_1, \alpha_2, \dots)$ of the entries of $\lambda$ and $x^\alpha = x_1^{\alpha_1} x_2^{\alpha_2} \cdots$.
\item  The \emph{elementary symmetric function}:
\begin{align*}
e_n(\xx) &:= \sum_{i_1>\cdots>i_n} x_{i_1}\cdots x_{i_n} \quad \mbox{for $n\ge 1$ with $e_0(\xx)=1$, and} \\
e_\lambda(\xx) &:= e_{\lambda_1}(\xx) e_{\lambda_2}(\xx)\cdots \quad \mbox{for a partition $\lambda=(\lambda_1,\lambda_2,\dots)$}.
\end{align*}
\item The \emph{complete homogeneous symmetric function}:
\begin{align*}
h_n(\xx) &:= \sum_{i_1\le\cdots\le i_n} x_{i_1}\cdots x_{i_n} \quad \mbox{for $n\ge 1$ with $h_0(\xx)=1$, and} \\
h_\lambda(\xx) &:= h_{\lambda_1}(\xx) h_{\lambda_2}(\xx)\cdots \quad \mbox{for a partition $\lambda=(\lambda_1,\lambda_2,\dots)$}.
\end{align*}
\item The \emph{power sum symmetric function}:
\begin{align*}
p_n(\xx) &:= \sum_{i} x_{i}^n \quad \mbox{for $n\ge 1$ with $p_0(\xx)=1$, and} \\
p_\lambda(\xx) &:= p_{\lambda_1}(\xx) p_{\lambda_2}(\xx)\cdots \quad \mbox{for a partition $\lambda=(\lambda_1,\lambda_2,\dots)$}.
\end{align*}
\end{enumerate}
From the fact that $\{e_\lambda(\xx)\}_{\lambda\in \Par}$ forms a basis
for $\Sym$, $\{e_1(\xx), e_2(\xx),\dots\}$ are algebraically independent
and generate $\Sym = \mathbb{Q}[e_1(\xx),e_2(\xx),\dots]$.

Let
\[
H(t) = \prod_{i\ge 1} \frac{1}{1-x_i t} = \sum_{n\ge 0} h_n(\xx) t^n \qand E(t) = \prod_{i\ge 1} (1+x_i t) = \sum_{n\ge 0} e_n(\xx) t^n.
\]
By definition, we obtain $H(t)E(-t)=1$, which yields the relation
\begin{equation} \label{eq:relation_eh}
h_k(\xx) - e_1(\xx) h_{k-1}(\xx) + \cdots + (-1)^k e_k(\xx) = \delta_{k,0}
\end{equation}
between complete homogeneous symmetric functions and elementary symmetric functions. Similarly let
\[
P(t) = \sum_{i\ge 1} \frac{x_i}{1-x_i t} = \sum_{n\ge 0} p_{n+1}(\xx) t^n.
\]
It is easy to check that $P(t) = \dfrac{d}{dt} \log H(t) = H'(t) E(-t)$, so we get
\begin{equation} \label{eq:relation_peh}
p_k(\xx) = e_1(\xx) h_{k-1}(\xx) - 2 e_2(\xx) h_{k-2}(\xx) + \cdots + (-1)^{k-1} k e_k(\xx).
\end{equation}
for $k>0$.

The most significant basis for \( \Sym \) is the basis of \emph{Schur functions}.
For a partition \( \lambda \), the Schur function \( s_\lambda(\xx) \)
is defined to be
\begin{equation} \label{eq:schur_ordinary_lattice_path}
s_\lambda(\xx) := \det (e_{\lambda'_i-i+j}(\xx))_{i,j=1}^{\lambda_1} = \det
\begin{pmatrix}
e_{\lambda'_1}(\xx) & e_{\lambda'_1+1}(\xx) & e_{\lambda'_1+2}(\xx) & \cdots \\
e_{\lambda'_2-1}(\xx) & e_{\lambda'_2}(\xx) & e_{\lambda'_2+1}(\xx) & \cdots \\
\vdots & \vdots &   &  \ddots
\end{pmatrix},
\end{equation}
where \( \lambda' \) is the conjugate of \( \lambda \).
The identity \eqref{eq:schur_ordinary_lattice_path} is called the \emph{dual
Jacobi--Trudi identity.}
Although not immediately apparent from the definition, it is indeed true that
Schur functions form a basis.
In addition, Schur functions have a combinatorial description. To give the
description, we need some combinatorial notions. A \emph{Young tableau}
of shape \( \lambda \) is a filling of the Young diagram of \( \lambda \)
with positive integers. We may think of a Young tableau $T$ of shape
$\lambda$ as the array $(T_{ij})$ of positive integers of shape $\lambda$ such
that $T_{ij}$ is the integer placed in the box $(i,j)$ in the matrix convention.
A Young tableau $T$ is called \emph{semistandard} if each row is weakly increasing
from left to right, and each column is strictly increasing from top to bottom.
Schur functions $s_\lambda(\xx)$ can be written purely combinatorially via
semistandard Young tableaux:
\begin{equation} \label{eq:schur_ordinary_tableau}
s_\lambda(\xx) = \sum_{T} x^T,
\end{equation}
where the sum ranges over all semistandard Young tableaux of shape $\lambda$ and $x^T=\prod_i x_i^{m_i(T)}$ where $m_i(T)$ is the number of appearances of $i$ in $T$.
There is a beautiful combinatorial proof via non-intersecting lattice paths that
two definitions \eqref{eq:schur_ordinary_lattice_path} and
\eqref{eq:schur_ordinary_tableau} are equivalent. We do not write down the
proof here, but we will revisit this method in Section~\ref{subsec:Noncomm_Schur}.

For a basis $\{b_\lambda(\xx)\}_{\lambda\in\Par}$ for $\Sym$,
a symmetric function $g(\xx)\in\Sym$ is said to be \emph{$b$-positive}
if the expansion of $g(\xx)$ in terms of the basis
$\{b_\lambda(\xx)\}_{\lambda\in\Par}$ has nonnegative coefficients.
Moreover, letting $\Sym_q = \mathbb{Q}[q]\otimes_\mathbb{Q} \Sym$,
$g(\xx)\in\Sym_q$ is said to be \emph{$b$-positive} if for any partition $\lambda$,
the coefficient of $b_\lambda(\xx)$ in the expansion of $g(\xx)$ with respect to
the basis $\{b_\lambda(\xx)\}_{\lambda\in\Par}$ is a polynomial in $q$ with
nonnegative coefficients.

We next introduce an involution $\omega$ on $\Sym$. Define an endomorphism
$\omega:\Sym\rightarrow\Sym$ by $\omega(e_k(\xx)) = h_k(\xx)$ for each $k\ge 1$,
and extend to $\Sym=\mathbb{Q}[e_1(\xx),e_2(\xx),\dots]$.
One can quickly check that $\omega$ is an involution,
i.e., $\omega^2 = id_{\Sym}$. Also, it is well known that
\[
\omega(e_\lambda(\xx)) = h_\lambda(\xx),\quad \omega(p_\lambda(\xx)) = (-1)^{|\lambda|-\ell(\lambda)} p_\lambda(\xx), \qand \omega(s_\lambda(\xx)) = s_{\lambda'}(\xx).
\]
The images $\omega(m_\lambda(\xx))$ of monomial symmetric functions are called
the \emph{forgotten symmetric functions} denoted by $f_\lambda(\xx)$.

Let $\yy=(y_1,y_2,\dots)$ be another set of commuting variables. Define
\[
C(\xx,\yy) := \prod_{i,j} \frac{1}{1-x_iy_j} \in\mathbb{Q}[[\xx,\yy]],
\]
called the \emph{Cauchy product}. From the definition of $C(\xx,\yy)$, we have
\begin{equation} \label{eq:comm_Cauchy_mh}
C(\xx,\yy) = \prod_{j\ge 1} \sum_{\ell\ge 0} x_j^\ell h_\ell(\yy)
= \sum_{\lambda\in\Par} m_\lambda(\xx) h_\lambda(\yy).
\end{equation}
Among many results related to $C(\xx,\yy)$, what we need is a list of various
expressions. We state them in here without proofs (see \cite{EC2}):

\begin{align}
C(\xx,\yy)
&= \sum_\lambda f_\lambda(\xx) e_\lambda(\yy) \label{eq:comm_Cauchy_ef} \\
&= \sum_\lambda \frac{1}{z_\lambda} p_\lambda (\xx) p_\lambda(\yy) \label{eq:comm_Cauchy_pp} \\
&= \sum_\lambda s_\lambda(\xx) s_\lambda(\yy) \label{eq:comm_Cauchy_ss} \\
&= \sum_\lambda h_\lambda(\xx) m_\lambda(\yy), \label{eq:comm_Cauchy_hm}
\end{align}
where each summation runs over all partitions, and $z_\lambda = \prod_{i\ge 1} i^{m_i} m_i !$ where $m_i$ is the number of parts of $\lambda$ which are equal to $i$.

We turn our attention to quasisymmetric functions. A \emph{quasisymmetric function} is a bounded degree formal power series in $\mathbb{Q}[[\xx]]$ such that for all compositions $\alpha=(\alpha_1,\dots,\alpha_\ell)$, the coefficient of $\prod x_i^{\alpha_i}$ equals the coefficient of $\prod x_{i_j}^{\alpha_i}$ for all $i_1<\dots <i_\ell$. Let $\QSym_n$ denote the space of quasisymmetric functions of degree $n$, and $\QSym = \bigoplus_{n\ge 0} \QSym_n$. By definition, $\Sym_n\subset\QSym_n$ and $\Sym\subset\QSym$.
For a composition $\alpha=(\alpha_1,\dots,\alpha_\ell)$, define the \emph{monomial quasisymmetric function} $M_{\alpha}(\xx)$ by
\[
  M_{\alpha}(\xx) := \sum_{i_1<\dots<i_\ell} x_{i_1}^{\alpha_1}\cdots x_{i_\ell}^{\alpha_\ell}.
\]
Then $\{ M_\alpha(\xx)\in\QSym_n \mid \alpha\vDash n\}$ forms a basis for
$\QSym_n$. Another basis for \( \QSym \) is the basis of
\emph{fundamental quasisymmetric functions}.
For a composition \( \alpha\vDash n \), the fundamental quasisymmetric
function $F_\alpha(\xx)$ is defined as
\[
  F_{\alpha}(\xx) := \sum_{\substack{i_1\le\cdots\le i_n \\ j\in S(\alpha)\Rightarrow i_j<i_{j+1}}} x_{i_1} x_{i_2}\cdots x_{i_n},
\]
often denoted by $F_{n,S(\alpha)}(\xx)$, where \( S \) is the map from
compositions to subsets of \( [n-1] \) defined in the previous section.
We can extend the involution $\omega$ to $\QSym$ as follows: for a subset $S\subset [n-1]$, $\omega F_{n,S}(\xx) := F_{n,[n-1]\setminus S}(\xx)$.

\subsection{Chromatic quasisymmetric functions}
In this paper, we deal with a slight generalization of chromatic quasisymmetric
functions introduced in Section~\ref{sec:intro}.
Let $G$ be a simple graph on the vertex set $[n]$,
and \( \mu=(\mu_1,\mu_2,\dots,\mu_n)\in\mathbb{N}^n \).
A \emph{proper multi-coloring} $\kappa$ of $G$ of type $\mu$ is a function from
$[n]$ to the collection of all finite subsets of $\mathbb{P}$ such that for each
$i\in [n]$, $| \kappa(i) | = \mu_i$ and $\kappa(j)\cap\kappa(k)=\emptyset$
whenever $\{j,k\}\in E(G)$.
The \emph{(multi-)chromatic quasisymmetric function} $X_G(\xx,q;\mu)$ of $G$
is given by
\[
  X_G(\xx,q;\mu) = \sum_{\kappa} q^{\asc_G(\kappa)} x^{\kappa},
\]
where the sum is over all proper multi-colorings $\kappa$ of type $\mu$,
$x^{\kappa} = \prod_{i=1}^n \prod_{k\in\kappa(i)} x_k$ and
$\asc_G(\kappa) = | \{ ((i,r),(j,s)) \mid \{i,j\}\in E(G),~ i<j,~ r\in\kappa(i),~ s\in\kappa(j),~ r<s \} |$.
See Figure~\ref{fig:multi_coloring}.
The multi-coloring generalization \( X_G(\xx,1;\mu) \) of chromatic symmetric
functions was introduced in \cite{Gasharov1996, Stanley1998}.
Obviously, $X_G(\xx,q) = X_G(\xx,q;(1^n))$.
\begin{figure}
\centering
\includegraphics[width=0.275\linewidth]{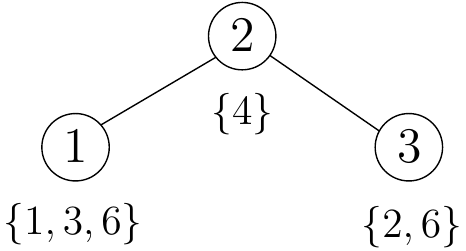}
\caption{For a graph $G=1-2-3$, let $\kappa$ be the proper multi-coloring of $G$ given by $\kappa(1)=\{1,3,6\}$, $\kappa(2)=\{4\}$ and $\kappa(3)=\{2,6\}$. Then $x^\kappa=x_1 x_2 x_3 x_4 x_6^2$ and $\asc_G(\kappa)=3$.} \label{fig:multi_coloring}
\end{figure}

In general, $X_G(\xx,q;\mu)$ is a quasisymmetric function, but when $G$ is the
incomparability graph of a natural unit interval order, $X_G(\xx,q;\mu)$ becomes
a symmetric function (Theorem~\ref{thm:X_P=Omega}).
Moreover, when $X_G(\xx,q;\mu)$ is symmetric, the symmetric function can be
expressed in an alternative way.
For a proper multi-coloring \( \kappa \) of \( G \), let
$\des_G(\kappa) =  | \{ ((i,r),(j,s)) \mid \{i,j\}\in E(G),~ i<j,~ r\in\kappa(i),~ s\in\kappa(j),~ r>s \} |$.
Then we have the following proposition, which follows immediately from the symmetry
of \( X_G(\xx,q;\mu) \).
\begin{prop}
If $X_G(\xx,q;\mu)$ is symmetric, then
\[
X_G(\xx,q;\mu) = \sum_{\kappa} q^{\des_G(\kappa)} x^\kappa,
\]
where $\kappa$ ranges over all proper multi-colorings of type $\mu$.
\end{prop}
By the proposition, for a natural unit interval order $P$, we have
\[
X_P(\xx,q;\mu) = \sum_{\kappa} q^{\des_P(\kappa)} x^\kappa.
\]

\begin{rmk} \label{rmk:multi-chrom}
Our multi-chromatic generalization $X_G(\xx,q;\mu)$ is equal, up to a scalar
multiple, to $X_{G^\mu}(\xx,q)$ of $G^\mu$ where $G^\mu$ is the graph obtained from
$G$ by replacing each vertex $a\in [n]$ by a complete graph of size $\mu_a$.
More precisely, let $d=\mu_1+\cdots+\mu_n$. For a vertex $i$, let $a_i$ be an
integer such that $\mu_1+\dots+\mu_{a_i-1}+1 \le i \le \mu_1+\dots+\mu_{a_i}$.
Then $G^{\mu}$ denotes the graph whose vertex set is $[d]$ and two vertices $i$ and
$j$ are adjacent if either $a_i=a_j$ or $a_i$ and $a_j$ are adjacent in $G$.
Therefore one can easily see that
\[
X_{G^\mu}(\xx,q) = \prod_{i=1}^n [\mu_i]_q! X_G(\xx,q;\mu),
\]
where $[k]_q=1+q+\cdots+q^{k-1}$ and \( [k]_q! = [1]_q\cdots [k]_q \).
Furthermore, for a natural unit interval order \( P \), the graph \( P^\mu \)
is also the incomparability graph of a certain natural unit interval order.
For example, let \( P=P(2,3,3) \) and \( \mu=(1,1,2) \). In this case,
\( P^\mu=(2,4,4,4) \) and we have
\[
  X_{P^\mu}(\xx,q) = [2]_q! X_P(\xx,q;\mu).
\]
\end{rmk}

For a natural unit interval order $P$, expansions of $X_P(\xx,q)$ into certain bases have been studied in a number of papers using various methods; $F$-expansion~\cite{SW2016}, $p$-expansion~\cite{Athanasiadis2015}, $s$-expansion~\cite{Gasharov1996,SW2016}. Also the $e$-expansion of $X_P(\xx,q)$ of certain natural unit interval orders $P$ has been studied in \cite{CH2019, HNY2020, HP2019, SW2016}. We close this section with an explicit statement of the $e$-positivity conjecture.
\begin{conj}[\cite{SS1993,Stanley1995,SW2016}] \label{conj:original_e_positivity_conj}
For a natural unit interval order $P$ on $[n]$, let
\[
  X_P(\xx,q) = \sum_{\lambda\vdash n} c_\lambda(q) e_\lambda(\xx).
\]
Then each $c_\lambda(q)$ is a polynomial with nonnegative integer coefficients.
\end{conj}

\section{Heaps and local flips} \label{sec:heap_local_flip}
In this section, we review the definition of heaps and define
an operation, called a local flip, on heaps. This operation plays a central role
in this paper. The theory of heaps, invented by Viennot~\cite{Viennot1986},
has been developed for sets equipped with a symmetric and reflexive relation,
but in this paper we restrict this theory to natural unit interval orders.
For another approach via heap theory, we refer to \cite{BN2020}.

Fix a natural unit interval order $P$ on $[n]$, and a nonnegative integer sequence
$\mu=(\mu_1,\dots,\mu_n)\in\mathbb{N}^n$. Similar to $G^\mu$ in
Remark~\ref{rmk:multi-chrom}, let $P^\mu$ be the graph whose vertex set is
$\{ v_{a,i} \mid a\in [n], 1\le i\le \mu_a\}$, and $v_{a,i}$ and $v_{b,j}$
are adjacent if either $a=b$ or $a$ and $b$ are incomparable in $P$.
A \emph{heap $H$} of $P$ of type $\mu$ is an acyclic orientation of $P^\mu$ satisfying that for each $a\in [n]$ and $1\le i< j\le \mu_a$,
the direction on the edge between $v_{a,i}$ and $v_{a,j}$ is toward $v_{a,i}$.
We call a vertex of a heap a \emph{block}, and denote the set of heaps of $P$ of type $\mu$ by $\HH(P,\mu)$.
It is worth noting that a heap of type \( (1^n) \) is just an acyclic orientation
of the graph \( P \). This parallels the fact that a proper multi-coloring of
type \( (1^n) \) is an ordinary proper coloring; see also
Remark~\ref{rmk:multi-chrom}.
We identify each heap \( H \) of \( P \) with the following diagram:
take $n$ unit intervals described in Section~\ref{subsec:NUIO}.
Let \( v=v_{a,1} \) be a sink of \( H \). Then stack a unit-length block on the
\( a \)th interval, and delete the vertex \( v \) from \( H \).
Again, choose a sink of \( H-v \), stack a unit-length block on the corresponding
interval, and delete the sink from \( H-v \). By repeating this process until
all vertices are deleted, we obtain the desired diagram.
For example, a heap of $P=P(2,4,5,5,5)$ of type $(3,1,0,2,2)$ and
the corresponding diagram are depicted in Figure~\ref{fig:ex_heap}.
\begin{figure}
\centering
\includegraphics[width=0.87\textwidth]{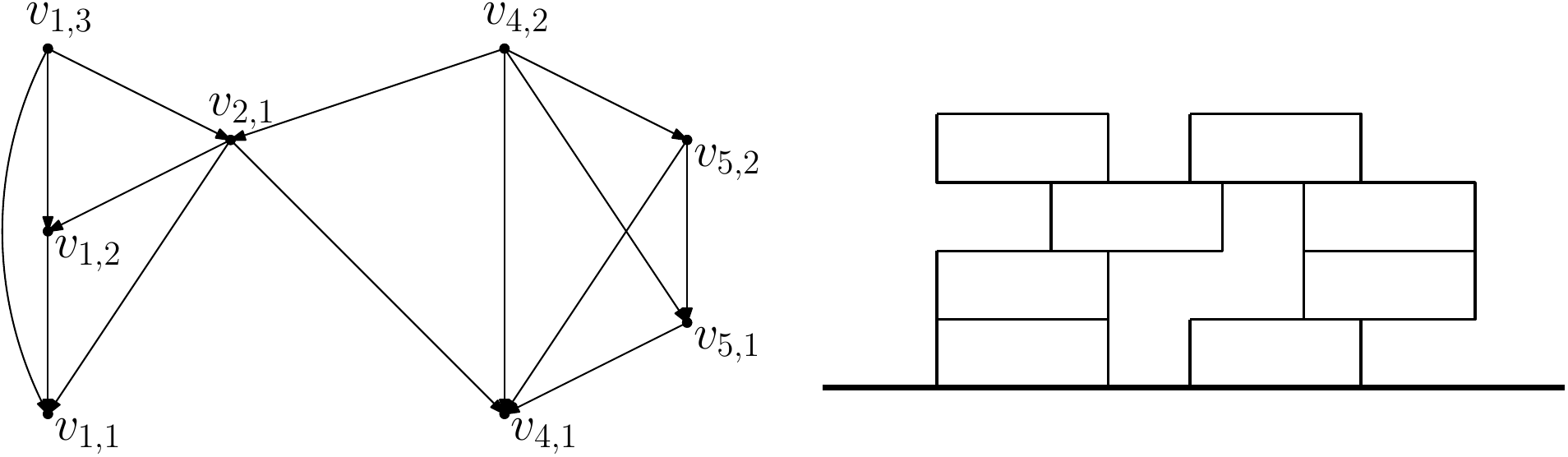} 
\caption{A heap of $P(2,4,5,5,5)$ of type $(3,1,0,2,2)$ and the corresponding diagram.} \label{fig:ex_heap}
\end{figure}

Each proper multi-coloring $\kappa$ of $P$ of type $\mu$ gives us a heap of $P$ of type $\mu$. For each vertex $a\in V(P)$, list its colors $\kappa(a)=\{c_{a,1}< \dots < c_{a,\mu_a}\}$, and for each edge $\{v_{a,i},v_{b,j}\}$ assign the direction $v_{a,i}\leftarrow v_{b,j}$ if $c_{a,i}<c_{b,j}$. Then this orientation is acyclic, and thus we have the heap of type $\mu$ corresponding to $\kappa$.
To see this diagrammatically, consider again $n$ unit intervals on the real line. For each $a\in [n]$ and $c\in\kappa(a)$, place a block at the position of height $c$ above the $a$th interval.
Due to the properness of \( \kappa \), there is no overlap between any pair of blocks. Once all the blocks are placed, let them descend as far as gravity allows. The resulting diagram precisely corresponds to the heap.
Consider Figure~\ref{fig:kappa} as an illustrative example.
Let $P=P(2,3,3)$ and $\kappa$ be a proper multi-coloring of $P$ of type $(3,1,2)$:
$\kappa(1) = \{1,3,6\}$, $\kappa(2) = \{4\}$, $\kappa(3) = \{2,6\}$, as depicted on the left. The corresponding heap is shown on the right.
\begin{figure}
\centering
\subfigure[]{
    \begin{tikzpicture}[scale=.45]
        \draw (0,0) rectangle ++(\RECTX,\RECTY);
        \draw (0,2) rectangle ++(\RECTX,\RECTY);
        \draw (0,5) rectangle ++(\RECTX,\RECTY);
        \draw (1.5,3) rectangle ++(\RECTX,\RECTY);
        \draw (3,1) rectangle ++(\RECTX,\RECTY);
        \draw (3,5) rectangle ++(\RECTX,\RECTY);
        \draw [thick] (-1,0) -- (6.5,0);
    \end{tikzpicture}
\label{fig:kappa_1}
}
\qquad
\subfigure[]{
    \begin{tikzpicture}[scale=.5]
        \draw (0,0) rectangle ++(\RECTX,\RECTY);
        \draw (0,1) rectangle ++(\RECTX,\RECTY);
        \draw (0,3) rectangle ++(\RECTX,\RECTY);
        \draw (1.5,2) rectangle ++(\RECTX,\RECTY);
        \draw (3,0) rectangle ++(\RECTX,\RECTY);
        \draw (3,3) rectangle ++(\RECTX,\RECTY);
        \draw [thick] (-1,0) -- (6.5,0);
    \end{tikzpicture}
\label{fig:kappa_2}
}
\caption{}
\label{fig:kappa}
\end{figure}

In addition, we can regard heaps as posets. By denoting a partial order on \( H \)
by \( \prec \), we define $v_{a,i}\prec v_{b,i}$ whenever $v_{a,i}\leftarrow v_{b,j}$, and then take their transitive closure.
From this point of view, a proper multi-coloring $\kappa$ of $P$ corresponding to $H$ can be thought of as a strict order-preserving map from $H$ to $\mathbb{P}$, that is, $\kappa:H\rightarrow\mathbb{P}$ satisfying $\kappa(v_{a,i}) < \kappa(v_{b,j})$ whenever $v_{a,i}\prec v_{b,j}$ in $H$. Hence the theory of $(P,\omega)$-partitions helps us study the chromatic quasisymmetric functions. 
First, we review this theory briefly.
(Previously, the letters $P$ and $\omega$ have been assigned to natural unit interval orders and the involution on $\Sym$, respectively. To prevent confusion, we use $(Q,\sigma)$ instead of the more conventional notation $(P,\omega)$.)
Let $Q$ be an arbitrary poset consisting of $n$ elements and $\sigma$ a bijection $\sigma:Q\rightarrow [n]$, called a \emph{labeling}. For a function $f:Q\rightarrow\mathbb{P}$, let $x^f = \prod_{t\in Q} x_{f(t)}$.
Then define the $(Q,\sigma)$-partition generating function by
\[
  K_{Q,\sigma}(\xx) = \sum_f x^f,
\]
where $f$ ranges over all functions $f:Q\rightarrow [n]$ satisfying
\begin{enumerate}[label=(\roman*)]
\item if $s\le t$ in $Q$, then $f(s)\le f(t)$,
\item if $s<t$ in $Q$ and $\sigma(s)>\sigma(t)$, then $f(s)<f(t)$.
\end{enumerate}
By definition, when $\sigma$ is order-preserving (respectively, order-reversing), $K_{Q,\sigma}(\xx)$ is the generating function for order-preserving (respectively, strictly order-preserving) maps of $P$. For simplicity, we write $K_Q(\xx)=K_{Q,\sigma}(\xx)$ for some order-preserving labeling $\sigma$, and $\overline{K}_Q(\xx) = K_{Q,\tau}(\xx)$ for some order-reversing labeling $\tau$.
By definition, $K_{Q,\sigma}(\xx)$ is a quasisymmetric function, and due to Stanley~\cite{EC2}, $K_Q(\xx) = \omega \overline{K}_Q(\xx)$. Let $\LL(Q,\sigma)$ be the set of permutations $\pi=\pi_1\dots\pi_n$ of $[n]$ such that the map $w:Q\rightarrow [n]$ defined by $w(\sigma^{-1}(\pi_i))=i$ is a linear extension of $Q$.
\begin{thm}[\cite{EC2}] \label{thm:Pw-partition}
We have
\[
K_{Q,\sigma}(\xx) = \sum_{\pi\in\LL(Q,\sigma)} F_{n,\Des(\pi)}(\xx),
\]
where $\Des(\pi)=\{i\in [n-1]\mid \pi_i > \pi_{i+1} \}$.
\end{thm}
We now apply Theorem~\ref{thm:Pw-partition} to the chromatic quasisymmetric functions. 
Let $W(\mu)$ be the set of all words of type $\mu$. For a heap $H$ of $P$ of type $\mu$, let $f:H\rightarrow [d]$ be a linear extension of $H$ where $d=\mu_1+\cdots+\mu_n=|H|$. Then define the word $\ww_f=\ww_{f,1}\cdots\ww_{f,d}$ by $\ww_{f,k} = a$ if $f^{-1}(k) = v_{a,i}\in H$ for some $i$. Then $\ww_f$ is of type $\mu$.
Let
\[
W(H) = \{\ww_f\in W(\mu) \mid \mbox{$f$ is a linear extension of $H$} \}.
\]
Then $\ww\in W(H)$ if and only if when blocks are piled on the $\ww_i$th interval
in order, the resulting diagram corresponds to $H$. Of course, $W(\mu)$ is the
disjoint union of $W(H)$'s, i.e., $W(\mu) = \bigsqcup_H W(H)$ where $H$ ranges
over all heaps of type $\mu$.
\begin{exam} \label{ex:W(H)}
Let $P=P(2,3,3)$ and $H$ be the heap depicted in Figure~\ref{fig:kappa_2}.
Let $f$ be a linear extension of $H$ given by $(f^{-1}(1),\dots,f^{-1}(6)) = (v_{1,1}, v_{3,1}, v_{1,2}, v_{2,1}, v_{1,3}, v_{3,2})$.
Then $\ww_f = \mathsf{131213}$. Also one can easily see
\[
W(H) = \{\mathsf{113213, 131213, 311213, 113231, 131231, 311231}\}.
\]
\end{exam}

For a heap $H$, an edge between $v_{a,i}$ and $v_{b,j}$ is \emph{ascent}
if the edge is toward $v_{a,i}$ and $a>b$ in the natural order.
Let $\asc_P(H)$ denote the number of ascent edges in $H$. We note that for a
multi-coloring $\kappa$, $\des_P(\kappa)=\asc_P(H)$ where $H$ is the heap
corresponding to $\kappa$.

\begin{thm} \label{thm:F_expan}
  We have
  \begin{align*}
    \omega X_P(\xx,q;\mu)
    &= \sum_{H\in \HH(P,\mu)} q^{\asc_P(H)} K_H(\xx) \\
    &= \sum_{\ww\in W(\mu)} q^{\inv_P(\ww)} F_{d,\Des_P(\ww)}(\xx).
  \end{align*}
\end{thm}
In \cite{SW2016}, Shareshian and Wachs established the theorem for the case where \( \mu=(1^n) \).
\begin{proof}
The relation between proper multi-colorings and strict order-preserving maps on $H$ implies that
\[
X_P(\xx,q;\mu) = \sum_{H\in\HH(P,\mu)} q^{\asc_P(H)} \overline{K}_{H}(\xx).
\]
Applying $\omega$ on both sides gives the first equation.
It is easy to check that $\asc_P(H) = \inv_P(\ww)$ for all $\ww\in W(H)$.
To prove the second equation, it is therefore enough to show that for any heap $H$,
\[
K_H(\xx) = \sum_{\ww\in W(H)} F_{d,\Des_P(\ww)}(\xx).
\]
Regarding $H$ as a poset, we construct a labeling $\sigma$ of $H$ as follows.
First, list all minimal elements $\{v_{a_1,1}, v_{a_2,1},\dots \}$ satisfying $a_1 <_P a_2 <_P \cdots$. Then label $v_{a_1,1}$ with $1$, and let $H'$ be the heap obtained from $H$ by removing the vertex $v_{a_1,1}$ and edges attached to $v_{a_1,1}$. Again list all minimal elements of $H'$, and label the smallest element (with respect to the ordering $<_P$ on $P$) with 2. Repeat this process until all blocks of $H$ are labeled. Figure~\ref{fig:labeling} might help in understanding the
labeling \( \sigma \). An integer in each block represents its label.
\begin{figure}[h]
\centering
\begin{tikzpicture}[scale=.5]
    \draw (0,0) rectangle ++(\RECTX,\RECTY);
    \draw (0,1) rectangle ++(\RECTX,\RECTY);
    \draw (0,3) rectangle ++(\RECTX,\RECTY);
    \draw (1.5,2) rectangle ++(\RECTX,\RECTY);
    \draw (3,0) rectangle ++(\RECTX,\RECTY);
    \draw (3,3) rectangle ++(\RECTX,\RECTY);

    \node at (1.25, 0.5) (1) {1};
    \node at (1.25, 1.5) (2) {2};
    \node at (4.25, 0.5) (3) {3};
    \node at (2.75, 2.5) (4) {4};
    \node at (1.25, 3.5) (5) {5};
    \node at (4.25, 3.5) (6) {6};

    \draw [thick] (-1,0) -- (6.5,0);
\end{tikzpicture}
\caption{}
\label{fig:labeling}
\end{figure}
\\
Then our labeling $\sigma$ is order-preserving by construction, so Theorem~\ref{thm:Pw-partition} yields that
\[
K_H(\xx) = \sum_{\pi\in\LL(H,\sigma)} F_{d,\Des(\pi)}(\xx).
\]
Define a map $\phi$ on $[d]$ by $\phi(k) = a$ if $\sigma^{-1}(k) = v_{a,i}$ for some $i$.
By abuse of notation we define a map $\phi$ from $\LL(H,\sigma)$ to $W(H)$: for each $\pi=\pi_1\dots\pi_d\in\LL(H,\sigma)$, define $\phi(\pi) = \phi(\pi_1)\phi(\pi_2)\dots\phi(\pi_d)\in W(H)$.
It is not hard to see that this map $\phi:\LL(H,\sigma)\rightarrow W(H)$ is bijective, and $\Des(\pi) = \Des_P(\phi(\pi))$, which completes the proof. For instance, let $H$ be the heap shown in Figure~\ref{fig:labeling}. Then
\[
\LL(H,\sigma) = \{\mathsf{123456, 132456, 312456, 123465, 132465, 312465}\}.
\]
Compare $\LL(H,\sigma)$ with $W(H)$ in Example~\ref{ex:W(H)}. Also we see that $\phi(\mathsf{132456}) = \mathsf{131213}$ and $\Des(\mathsf{132456}) = \{2\} = \Des_P(\mathsf{131213})$.
\end{proof}

Note that Theorem~\ref{thm:F_expan} holds not only for natural unit interval orders, but for arbitrary posets on $[n]$ under a suitable setting.

We are now in a position to define an important operation on heaps.
For any distinct blocks $p,q$ and $r$ in $H$, we call $(p,q,r)$ a
\emph{flippable triple} in $H$ if regrading \( H \) as a poset, one of the
following conditions hold:
\begin{enumerate}[label=(\roman*)]
\item $q$ covers $p$ and $r$;
\item $q$ is covered by $p$ and $r$.
\end{enumerate}
By definition, $p$ and $r$ are not adjacent.
Note that each block covers at most two blocks, and is covered by at most two
blocks because of the $\mathbf{(3+1)}$-freeness of $P$.
\begin{defn}
Let $H$ be a heap of $P$ and $(p,q,r)$ a flippable triple in $H$.
A \emph{local flip} at $(p,q,r)$ is reversing the directions on the edges
$\{p,q\}$ and $\{q,r\}$.
\end{defn}
The following lemma follows easily from the definition of flippable triples.
\begin{lem}
Let $H$ be a heap and $(p,q,r)$ a flippable triple in $H$.
Then the orientation $H'$ obtained from $H$ by a local flip at $(p,q,r)$ is
acyclic, so $H'$ is also a heap of the same type.
\end{lem}

As we can see from the example below, we can think of a local flip as an operation
on diagrams of heaps acting by transposing relative positions of
blocks as follows:
\begin{center}
  \includegraphics[width=0.45\linewidth]{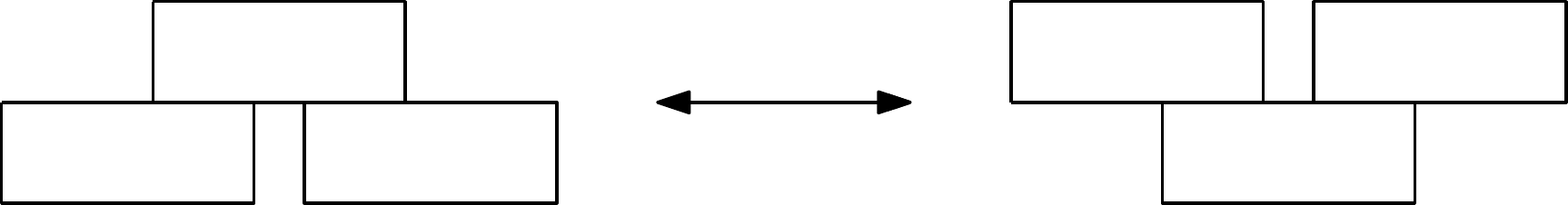}
\end{center}
\begin{exam} \label{ex:flip}
  Let $H$ be the heap shown in Figure~\ref{fig:ex_heap}.
  There are 4 flippable triples: $(v_{1,2}, v_{2,1}, v_{4,1}), (v_{1,3}, v_{2,1}, v_{4,2}), (v_{2,1}, v_{4,1}, v_{5,1})$ and $(v_{2,1}, v_{4,2}, v_{5,2})$.
  The following diagram is the heap obtained from $H$
  by a local flip at $(v_{2,1}, v_{4,2}, v_{5,2})$:
  \begin{center}
    \begin{tikzpicture}[scale=.5]
      \draw (0,0) rectangle ++(\RECTX,\RECTY);
      \draw (0,1) rectangle ++(\RECTX,\RECTY);
      \draw (0,4) rectangle ++(\RECTX,\RECTY);
      \draw (1.5,3) rectangle ++(\RECTX,\RECTY);
      \draw (3.5,0) rectangle ++(\RECTX,\RECTY);
      \draw (3.5,2) rectangle ++(\RECTX,\RECTY);
      \draw (5,1) rectangle ++(\RECTX,\RECTY);
      \draw (5,3) rectangle ++(\RECTX,\RECTY);
      \draw [thick] (-1,0) -- (8.5,0);
  \end{tikzpicture}
  \end{center}
\end{exam}

Using local flips, we can define an equivalence relation on the set of heaps of $P$ of type $\mu$: for two heaps $H,H'\in\HH(P,\mu)$, $H\sim H'$ if and only if $H'$ can be obtained from $H$ by applying a finite sequence of local flips. For instance, we illustrate all heaps of $P(2,3,4,5,5)$ of type $(1^5)$ and their equivalence relations in Figure~\ref{fig:heap_orbit_2345}.
\begin{figure}
\centering
\includegraphics[width=0.92\linewidth]{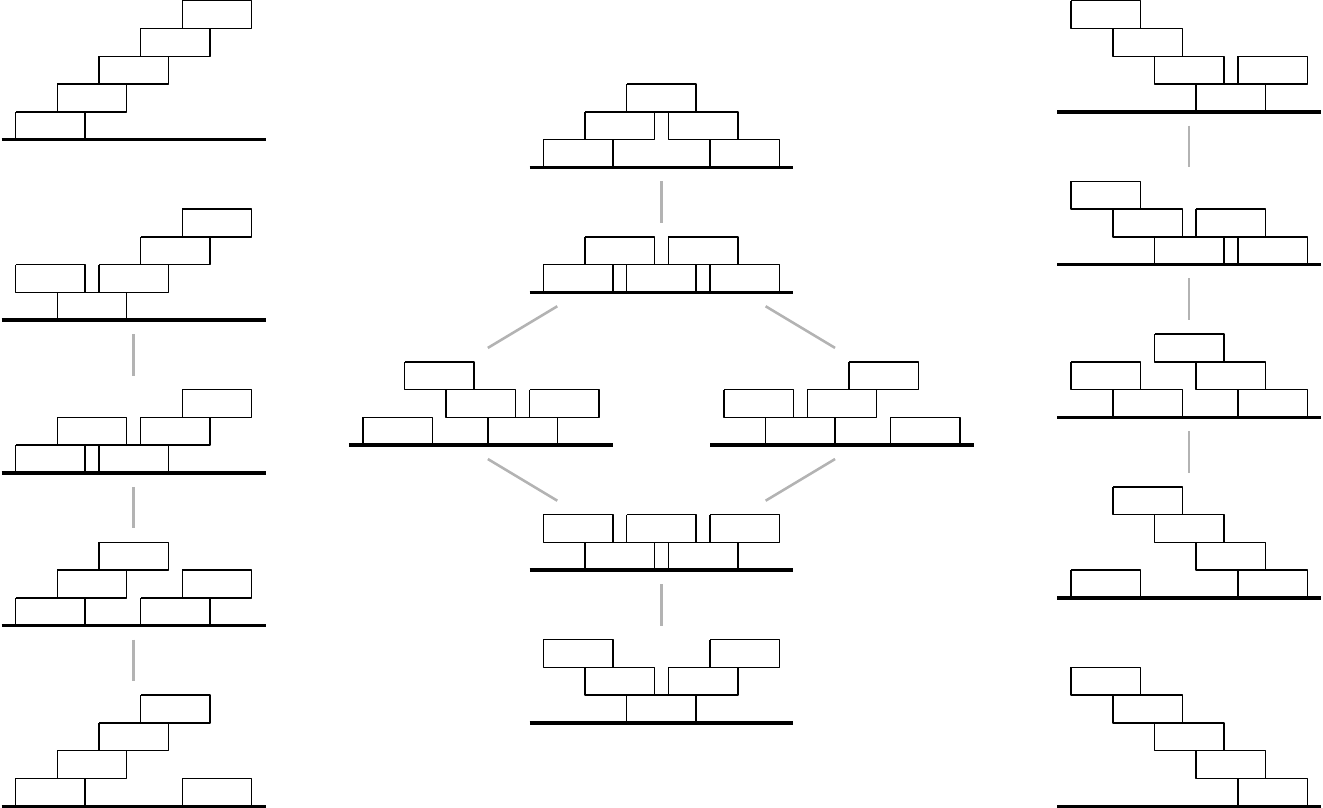}
\caption{All heaps of $P(2,3,4,5,5)$ of type $(1^5)$. A gray line between heaps means that they can be transformed to each other via a local flip. Hence each connected component represents an equivalence class.}
\label{fig:heap_orbit_2345}
\end{figure}

The following proposition and theorem tell us why local flips are crucial.
\begin{prop} \label{prop:ascent_number}
Local flips preserve the number of ascents, i.e., $\asc_P(H)=\asc_P(H')$ if $H\sim H'$.
\end{prop}
\begin{proof}
Let $(p,q,r)$ be a flippable triple in $H$. Also let $a,b,c$ be integers such that
$p=v_{a,i}$, $q=v_{b,j}$, $r=v_{c,k}$ for some $i,j,k$. We may assume that
$q\rightarrow p$, $q\rightarrow r$ in \( H \) and $a<_P c$.
Then one can easily see that $a<b<c$ in the natural order on $\mathbb{P}$,
and $a\nless_P b\nless_P c$.
Hence $q\rightarrow r$ contributes to $\asc_P(H)$, while $q\rightarrow p$ is not.
In $H'$, this two edges are reversed, so $p\rightarrow q$ contributes to
$\asc_P(H')$, while $r\rightarrow q$ is not. Other edges are unchanged,
so the number of ascents is preserved.
\end{proof}
For \( [H]\in\equivclass{\HH(P,\mu)} \), define \( \asc_P([H]):=\asc_P(H') \)
for any \( H'\in [H] \). The proposition guarantees the well-definedness.

\begin{thm} \label{thm:X_P_sym}
  Let $[H]$ be an equivalence class in $\equivclass{\HH(P,\mu)}$. Then
  \[
  K_{[H]}(\xx) := \sum_{H'\in [H]} K_{H'}(\xx) \in \Sym.
  \]
  In particular, $X_G(\xx,q;\mu)$ is a symmetric function.
\end{thm}
Note that by definition, we have
\begin{equation} \label{eq:X_P=sum_K_[H]}
  \omega X_P(\xx,q;\mu) = \sum_{[H]\in\equivclass{\HH(P, \mu)}} q^{\asc_P([H])} K_{[H]}(\xx).
\end{equation}
The theorem follows from some results in Section~\ref{sec:Noncomm},
and thus we postpone the proof until all necessary materials are prepared.
\begin{rmk} \label{rmk:2+2}
    Recall that a poset \( P \) is a natural unit interval order if and only if
    \( P \) is \( \mathbf{(3+1)} \)-and-\( \mathbf{(2+2)} \)-free.
    Since there is no ``natural'' labeling on elements of a $\mathbf{(3+1)}$-free
    poset, the ascent statistic can not be extended to the class of
    $\mathbf{(3+1})$-free posets in a natural way. But we do not need the
    $\mathbf{(2+2)}$-freeness of $P$ for defining local flips, and hence local
    flips can be defined on heaps of the incomparability graphs of
    $\mathbf{(3+1)}$-free posets.
\end{rmk}

The equivalence relation is a more refined notion than
the ascent statistic introduced by Shareshian and Wachs.
Indeed, two heaps with the same ascent number could belong to different
equivalence classes, while Proposition~\ref{prop:ascent_number} says that
any two heaps $H,H'$ belonging to the same equivalent class have
the same ascent number; see Figure~\ref{fig:nonequiv_heaps}.
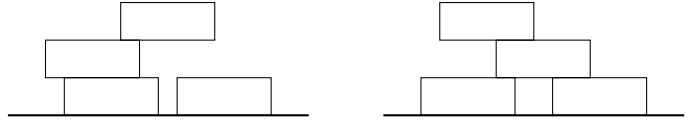
\begin{figure}
  \begin{tikzpicture}[scale=.5]
    \draw (  0,1) rectangle ++(\RECTX,\RECTY);
    \draw (0.5,0) rectangle ++(\RECTX,\RECTY);
    \draw (  2,2) rectangle ++(\RECTX,\RECTY);
    \draw (3.5,0) rectangle ++(\RECTX,\RECTY);
    \draw [thick] (-1,0) -- (7,0);
  \end{tikzpicture} \qquad
  \begin{tikzpicture}[scale=.5]
    \draw (  0,0) rectangle ++(\RECTX,\RECTY);
    \draw (0.5,2) rectangle ++(\RECTX,\RECTY);
    \draw (  2,1) rectangle ++(\RECTX,\RECTY);
    \draw (3.5,0) rectangle ++(\RECTX,\RECTY);
    \draw [thick] (-1,0) -- (7,0);
  \end{tikzpicture}
  \caption{Two non-equivalent heaps with the same ascent number;
    here \( P=(3,3,4,4) \) and heaps of type \( (1^4) \).}
  \label{fig:nonequiv_heaps}
\end{figure}
In addition, Theorem~\ref{thm:X_P_sym} admits a refinement of
the refined $e$-positivity conjecture.
\begin{conj} \label{conj:refined_e-positivity}
Let $P$ be a natural unit interval order on $[n]$, and $\mu\in\mathbb{N}^n$.
For any equivalence class $[H]\in\equivclass{\HH(P,\mu)}$, $K_{[H]}(\xx)$ is
$h$-positive.
In particular, $X_P(\xx,q;\mu)$ is $e$-positive.
\end{conj}

We end this section with constructing an edge-labeled graph $\Gamma_\mu$ for $\mu\in\mathbb{N}^n$ which will appear again in the next section.
The vertex set of $\Gamma_\mu$ is the set $W(\mu)$ of all words of type $\mu$. Two words $\ww=\ww_1\cdots\ww_d$ and $\vv=\vv_1\dots\vv_d$ are adjacent in $\Gamma_\mu$ if there exists an integer $i$ satisfying the one of the following conditions:
\begin{enumerate}[label=(\roman*)]
\item $\ww_j=\vv_j$ for $j\notin\{i,i+1\}$, and $\{\ww_i\ww_{i+1},\vv_i\vv_{i+1}\}=\{\mathsf{ac, ca}\}$ for some $a<_P c$. (Label this edge with $i$)
\item $\ww_j=\vv_j$ for $j\notin\{i-1,i,i+1\}$, and  either $\{\ww_{i-1}\ww_i\ww_{i+1},\vv_{i-1}\vv_i\vv_{i+1}\}=\{\mathsf{bac, acb}\}$ or $\{\mathsf{bca, cab}\}$  for some $a<b<c$ satisfying $a\nless_P b$, $b\nless_P c$ and $a<_P c$. (Label this edge with $\bar{i}$)
\end{enumerate}
The second condition represents how a local flip operates on words. Let $\Gamma$ be a disjoint union of $\Gamma_\mu$ for all $\mu$.

The graph $\Gamma$ is reminiscent of dual equivalence graphs \cite{Assaf2015} or
switchboards \cite{BF2017}. Each connected component of dual equivalence graphs or
switchboards defines a positive sum of specific Schur functions
(sometimes, a single Schur function) where they rely on the structure of
the connected component \cite{Assaf2015,BF2017}. In our case, $\Gamma$ exhibits a
similar property; see Theorem~\ref{thm:s-positive}.

The following proposition tells us a connection between connected components of $\Gamma_\mu$ and equivalence classes in $\equivclass{H(P,\mu)}$.
\begin{prop} \label{prop:gamma_connected_component}
Let $H,H'$ be heaps of type $\mu$. Then the following are equivalent:
\begin{enumerate}[label=(\roman*),font=\upshape]
\item For some $\ww\in W(H)$ and $\ww'\in W(H')$, $\ww$ and $\ww'$ are contained in the same connected component of $\Gamma_\mu$.
\item For all $\ww\in W(H)$ and $\ww'\in W(H')$, $\ww$ and $\ww'$ are contained in the same connected component of $\Gamma_\mu$.
\item $H\sim H'$.
\end{enumerate}
\end{prop}
\begin{proof}
It is a basic fact that $\ww,\ww' \in W(H)$ if and only if there is an path from
$\ww$ to $\ww'$ along only unbarred edges; see \cite[Lemma 4]{Viennot1986}.
Then let us consider barred edges. Suppose that $H$ and $H'$ can be transformed
to each other via a local flip at a flippable triple $(p,q,r)$.
Without loss of generality, let $p=v_{a,i}$, $q=v_{b,j}$ and $r=v_{c,k}$
for some $1\le a < b < c \le n$.
Also suppose that $q\rightarrow p$ and $q\rightarrow r$ in $H$. By the definition of a flippable triple, there is a word $\ww\in W(H)$ such that $\ww=\cdots \mathsf{acb} \cdots$.
Let $\ww' = \cdots \mathsf{bac} \cdots$, so one can quickly check that $\ww'\in W(H')$. But in this case, $\ww$ and $\ww'$ is connected by a barred edge.
Using these facts, the verification of this proposition is straightforward, and we
leave the details to the reader.
\end{proof}
Proposition~\ref{prop:gamma_connected_component} yields that there is a one-to-one
correspondence between equivalence classes in $\equivclass{\HH(P,\mu)}$ and
connected components of $\Gamma_\mu$. Moreover, letting $\Gamma'_\mu$ be the
subgraph of $\Gamma_\mu$ on the same vertex set containing unbarred edges only,
there is a one-to-one correspondence between heaps of type $\mu$ and connected
components of $\Gamma'_\mu$.
\begin{exam} \label{ex:Gamma}
Let $P=P(2,3,3)$ and $\mu=(1,1,2)$. Then $W(\mu)$ consists of 12 words.
Also there are 6 heaps of $P$ of type $\mu$ and 4 equivalence classes in
$\equivclass{\HH(P,\mu)}$; see Figure~\ref{fig:Gamma_112_heap}. The graph
$\Gamma_\mu$ shown in Figure~\ref{fig:Gamma_112_word} reflects this information.
We will return to this example later.
\begin{figure}[t]
\centering
\subfigure[]{
\includegraphics[width=0.7\linewidth]{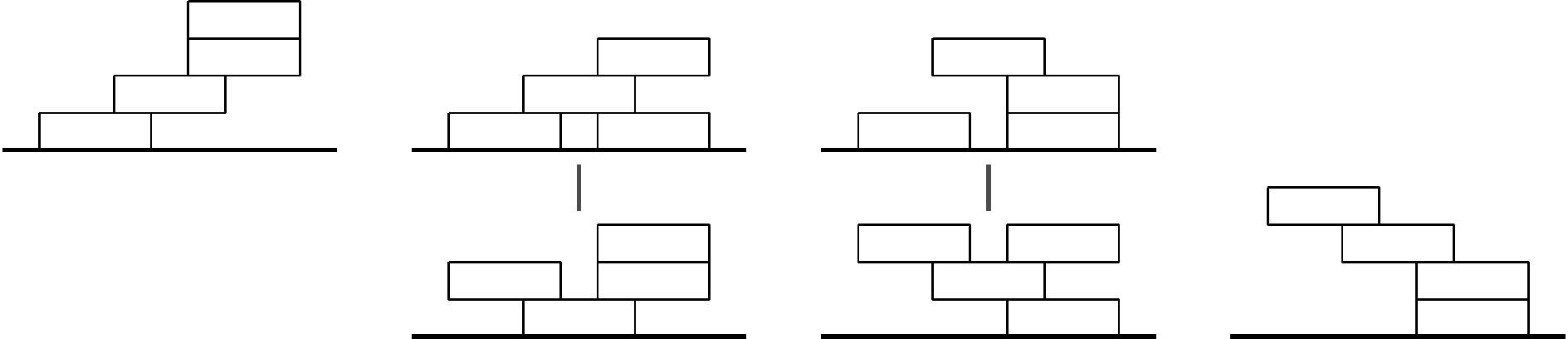}
\label{fig:Gamma_112_heap}
}
\subfigure[]{
\includegraphics[width=0.7\linewidth]{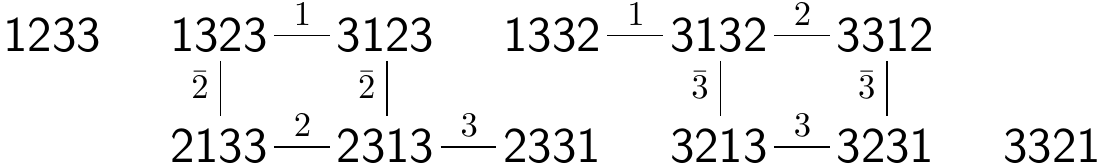}
\label{fig:Gamma_112_word}
}
\caption{Let $P=P(2,3,3)$ and $\mu=(1,1,2)$. (a) All heaps of $P$ of type $\mu$ and their equivalence relations. (b) The graph $\Gamma_\mu$.}\label{fig:Gamma_112}
\end{figure}
\end{exam}

\section{Noncommutative \texorpdfstring{$P$}{P}-symmetric functions} \label{sec:Noncomm}
In this section, we define noncommutative $P$-symmetric functions associated with
a natural unit interval order $P$, and present their connection
with the chromatic quasisymmetric function of \( P \).
Using these, we prove positivity phenomena of $X_P(\xx,q;\mu)$
in several symmetric function bases.

The theory of noncommutative symmetric functions was introduced by Fomin and
Greene~\cite{FG1998}, which is a framework for proving positivity of symmetric
functions. Using this, Schur expansions of various symmetric functions have
been discovered; see, e.g., \cite{FG1998, Blasiak2016}.
Unfortunately, this theory cannot be applied to chromatic quasisymmetric functions
directly. Hence we modify the general setting of the theory slightly,
and show several positivities of the chromatic quasisymmetric functions.
Our modification is inspired by \cite{Stanley1998}.

\subsection{An analogue of noncommutative symmetric functions and the Cauchy product}
Let $P$ be a natural unit interval graph on $[n]$ and $\UU=\mathbb{Z}\langle u_1,\dots,u_n \rangle$ the free associative $\mathbb{Z}$-algebra generated by $\{u_1,\dots,u_n\}$.
For simplicity we write $u_\ww=u_{\ww_1} u_{\ww_2} \cdots u_{\ww_d}$ for a word $\ww=\ww_1 \ww_2 \cdots \ww_d$ on the alphabet $[n]$.
Let $\II_P$ be the 2-sided ideal of $\UU$ generated by the following elements:
\begin{align}
& u_a u_c - u_c u_a & &\mbox{($a <_P c$),} \label{eq:generator_I1}\\
& u_a u_c u_b - u_b u_a u_c & &\mbox{($a<b<c$, $a\nless_P b$, $b\nless_P c$ and $a<_P c$).} \label{eq:generator_I2}
\end{align}
The ideal is just the algebraic counterpart of the graph $\Gamma$,
and hence the following proposition is essentially equivalent to Proposition~\ref{prop:gamma_connected_component}.
\begin{prop} \label{prop:algebraic_expression_local_flip}
  Let $H,H'$ be heaps. Then the following are equivalent:
  \begin{enumerate}[label=(\roman*),font=\upshape]
    \item $u_\ww \equiv u_{\ww'} \mod \II_P$ for some $\ww\in W(H)$ and $\ww'\in W(H')$.
    \item $u_\ww \equiv u_{\ww'} \mod \II_P$ for all $\ww\in W(H)$ and $\ww'\in W(H')$.
    \item $H\sim H'$.
  \end{enumerate}
\end{prop}

For $k\ge1$, we define the \emph{noncommutative $P$-elementary symmetric function} $\ee^P_k(\uu)$, which is not in general symmetric on variables $u_1,\dots,u_n$, by
\begin{equation} \label{eq:def_ee}
\ee^P_k(\uu) = \sum_{i_1 >_P i_2 >_P \cdots >_P i_k} u_{i_1} u_{i_2} \cdots u_{i_k} \in\UU.
\end{equation}
By convention, let $\ee^P_0(\uu) = 1$ and $\ee^P_k(\uu) = 0$ for any $k<0$.
For a partition $\lambda=(\lambda_1,\dots,\lambda_\ell)$, define
$\ee^P_\lambda(\uu) = \ee^P_{\lambda_1}(\uu)\cdots\ee^P_{\lambda_\ell}(\uu)$.
Throughout this paper, the given poset $P$ is always clear, so we write
$\ee_\lambda(\uu)$ instead of $\ee^P_{\lambda}(\uu)$.
Moreover, while other noncommutative $P$-symmetric functions, to be defined later,
also depend on \( P \), we will omit the superscript $P$.
Before introducing others noncommutative \( P \)-symmetric functions, we verify
the following important property.
Blasiak and Fomin~\cite{BF2017} called this property the commutation relation.
\begin{thm} \label{thm:e_k_commute}
For any integers $k, \ell\ge 0$, $\ee_k(\uu)$ and $\ee_\ell(\uu)$ commute with
each other modulo $\II_P$, that is,
\[
\ee_k(\uu) \ee_\ell(\uu) \equiv \ee_\ell(\uu) \ee_k(\uu) \mod \II_P.
\]
\end{thm}
\begin{proof}
For $m\ge 1$, let $\EE_m$ be the set of all words $\ww=\ww_1\cdots \ww_m$ such that $\ww_1>_P\cdots >_P \ww_m$. By convention, let $\EE_0$ be the set consisting of the empty word. Since
\[
\ee_k(\uu) \ee_\ell(\uu) = \sum_{(\ww,\vv)\in\EE_k\times\EE_\ell} u_\ww u_\vv,
\]
we prove this theorem by constructing a bijection $\psi_{k,\ell}: \EE_k \times \EE_\ell \rightarrow \EE_\ell \times \EE_k$ satisfying for $(\ww, \vv)\in \EE_k \times \EE_\ell$, $u_\ww u_\vv \equiv u_{\vv'} u_{\ww'}$ modulo $\II_P$ where $(\vv',\ww') = \psi(\ww,\vv)$.
We describe $\psi_{k,\ell}$ using diagrams which are similar to diagrams of heaps. Let $\ww=\ww_1\cdots \ww_k$ and $\vv=\vv_1\cdots \vv_\ell$. We take again $n$ unit intervals on the real line, which correspond to $P$. Then for $1\le i\le k$ put a block on the $\ww_i$th interval. Also for $1\le j\le \ell$ place a block at height $2$ above on the $\vv_j$th interval.
Then blocks may or may not be connected to others placed at different height. If a block placed at height $1$ is isolated, then lift it to height $2$, and vice versa if a block at height $2$ is isolated, then lower it to height $1$.
For non-isolated blocks, we consider their connected components. Each connected component consists of $r$ blocks at height $1$ and $s$ blocks at height $2$ with $|r-s| \le 1$.
If $r=s$, we leave them unchanged. If $r\neq s$, we switch their heights; lift all blocks at height $1$ to at height $2$, and drop all blocks at height $2$ to at height $1$. In fact, this process is a consequence of a series of local flips.
Hence, we finally obtain a diagram consisting of $\ell$ blocks at height $1$ and $k$ blocks at height $2$. Figure~\ref{fig:commute} shows how this procedure works. A diagram shown in Figure~\ref{fig:commute_1} corresponds to a pair $(\ww,\vv)\in\EE_5\times\EE_6$. The diagram obtained by applying the procedure described above to $(\ww,\vv)$ is depicted in Figure~\ref{fig:commute_2}.
\begin{figure}
\centering
\subfigure[]{
    \begin{tikzpicture}[scale=.5]
        \draw (0,1) rectangle ++(\RECTX,\RECTY);
        \draw (1,0) rectangle ++(\RECTX,\RECTY);
        \draw (4,0) rectangle ++(\RECTX,\RECTY);
        \draw (6,1) rectangle ++(\RECTX,\RECTY);
        \draw (7.5,0) rectangle ++(\RECTX,\RECTY);
        \draw (11,1) rectangle ++(\RECTX,\RECTY);
        \draw (14,1) rectangle ++(\RECTX,\RECTY);
        \draw (15.5,0) rectangle ++(\RECTX,\RECTY);
        \draw (17,1) rectangle ++(\RECTX,\RECTY);
        \draw (19,0) rectangle ++(\RECTX,\RECTY);
        \draw (20,1) rectangle ++(\RECTX,\RECTY);
        \draw [thick] (-1,0) -- (23.5,0);
    \end{tikzpicture}
\label{fig:commute_1}
}
\subfigure[]{
    \begin{tikzpicture}[scale=.5]
        \draw (0,1) rectangle ++(\RECTX,\RECTY);
        \draw (1,0) rectangle ++(\RECTX,\RECTY);
        \draw (4,1) rectangle ++(\RECTX,\RECTY);
        \draw (6,0) rectangle ++(\RECTX,\RECTY);
        \draw (7.5,1) rectangle ++(\RECTX,\RECTY);
        \draw (11,0) rectangle ++(\RECTX,\RECTY);
        \draw (14,0) rectangle ++(\RECTX,\RECTY);
        \draw (15.5,1) rectangle ++(\RECTX,\RECTY);
        \draw (17,0) rectangle ++(\RECTX,\RECTY);
        \draw (19,1) rectangle ++(\RECTX,\RECTY);
        \draw (20,0) rectangle ++(\RECTX,\RECTY);
        \draw [thick] (-1,0) -- (23.5,0);
    \end{tikzpicture}
\label{fig:commute_2}
}
\caption{}
\label{fig:commute}
\end{figure}

Then we let $\vv'$ and $\ww'$ be the words corresponding to blocks at height $1$ and $2$, respectively. One can deduce from Proposition~\ref{prop:algebraic_expression_local_flip} that $u_\ww u_\vv \equiv u_{\vv'} u_{\ww'}$ modulo $\II_P$, which finishes the proof.
\end{proof}

Mimicking the relation \eqref{eq:relation_eh}, we define the
\emph{noncommutative $P$-complete homogeneous symmetric functions}
$\hh_k(\uu)$ inductively as follows:
\[
\hh_k(\uu) - \ee_1(\uu) \hh_{k-1}(\uu) + \cdots + (-1)^k \ee_k(\uu) = \delta_{k,0},
\]
with $\hh_0(\uu)=1$, and define $\hh_\lambda(\uu) = \hh_{\lambda_1}(\uu)\cdots\hh_{\lambda_\ell}(\uu)$ for a partition $\lambda=(\lambda_1,\dots,\lambda_\ell)$.
Then it is easy to check that
\begin{equation} \label{eq:hh}
\hh_k(\uu) = \sum_{i_1\ngtr_P i_2\ngtr_P \cdots \ngtr_P i_k} u_{i_1} u_{i_2} \cdots u_{i_k}.
\end{equation}

\begin{prop} \label{prop:unique_word_and_hh}
Let $P$ be a natural unit interval order, and $H$ a heap of $P$.
Then there is a unique word $\ww\in W(H)$ such that $\ww$ has no $P$-descents.
Denoting this word by $\ww_H$, we have
\[
\hh_k(\uu) = \sum_{H} u_{\ww_H},
\]
where the sum is over all heaps $H$ consisting of $k$ blocks.
\end{prop}
\begin{proof}
In the proof of Theorem~\ref{thm:F_expan}, we constructed an order-preserving
labeling $\sigma$ together with a bijective map $\phi:\LL(H,\sigma)\rightarrow W(H)$ such that $\Des(\pi)=\Des_P(\phi(\pi))$ for all $\pi\in\LL(H,\sigma)$.
Since $\sigma$ is order-preserving, there exists a unique $\pi_0\in\LL(H,\sigma)$
with no descents. Then $\phi(\pi_0)$ is the unique word in $W(H)$ with no
$P$-descents.
The second part of this lemma follows immediately from \eqref{eq:hh}.
\end{proof}

By Theorem~\ref{thm:e_k_commute}, $\hh_k(\uu)$'s also commute with each other in $\UU/\II_P$ since they can be written as polynomials in $\ee_1(\uu), \ee_2(\uu), \dots$.

Define
\[
H(x,\uu) = \sum_{\ell\ge 0} x^\ell \hh_\ell(\uu) \in\UU[[x]]
\]
and
\[
\Omega(\xx, \uu) = H(x_1, \uu)H(x_2, \uu)\cdots = \sum_\alpha M_\alpha(\xx) \hh_\alpha(\uu) \in\UU[[\xx]],
\]
where $\alpha=(\alpha_1,\dots,\alpha_\ell)$ ranges over all compositions and $\hh_\alpha(\uu) = \hh_{\alpha_1}(\uu)\cdots\hh_{\alpha_\ell}(\uu)$. Here, $x$ and $\xx$ commute with $\uu$.
We call $\Omega(\xx,\uu)$ the \emph{noncommutative $P$-Cauchy product}.
Observing the definition of $\Omega(\xx,\uu)$ and \eqref{eq:hh}, we have
\begin{equation} \label{eq:Omega=F}
\Omega(\xx,\uu) = \sum_{\ww} F_{d, \Des_P(\ww)}(\xx) u_\ww,
\end{equation}
where $\ww$ ranges over all words on the alphabet $[n]$ and $d$ is the length of $\ww$.
Since $H(x_i,\uu) H(x_j,\uu) \equiv H(x_j,\uu) H(x_i,\uu)$ modulo $\II_P[[\xx]]$,
we can write $\Omega(\xx,\uu)$ as the usual Cauchy product \eqref{eq:comm_Cauchy_mh}:
\begin{equation} \label{eq:Omega=mh}
\Omega(\xx,\uu) \equiv \sum_{\lambda} m_\lambda(\xx) \hh_\lambda(\uu) \mod \II_P[[\xx]].
\end{equation}
The noncommutative $P$-Cauchy product $\Omega(\xx,\uu)$ gives us a duality
between the chromatic quasisymmetric function $X_P(\xx,q;\mu)$ and
noncommutative $P$-symmetric functions.
To explain what the term `duality' means, let $\UU^*$ be the free
$\mathbb{Z}$-module generated by words on the alphabet $[n]$, and
$\langle~,~\rangle$ a canonical pairing between $\UU$ and $\UU^*$ such that
$\langle u_\ww , \vv \rangle = \delta_{\ww,\vv}$ for words $\ww$ and $\vv$.
Then $\UU^*$ is the dual space of $\UU$ as $\mathbb{Z}$-modules. Letting
\[
\II_P^\perp = \{\gamma\in\UU^* \mid \langle z,\gamma \rangle = 0 \mbox{ for all } z\in\II_P\}
\]
be the orthogonal complement of $\II_P$, we have the naturally induced pairing between $\UU/\II_P$ and $\II_P^\perp$:
for $z+\II_P \in \UU/\II_P$ and $\gamma\in\II_P^\perp$, the pairing $\langle z+\II_P, \gamma \rangle := \langle z, \gamma \rangle$ is well-defined. Also let $\UU^*_q = \mathbb{Z}[q] \otimes_\mathbb{Z} \UU^*$ and $\II_{P,q}^\perp = \mathbb{Z}[q] \otimes_\mathbb{Z} \II_P^\perp$. Then we extend the pairing to $\UU_q^*$, and $\II_{P,q}^\perp$ plays the role of $\II_P^\perp$.

Let
\[
\gamma_{H} = \sum_{\ww\in W(H)} \ww \in\UU^* \qand \gamma_{[H]} = \sum_{H'\in [H]} \gamma_{H'} \in\UU^*
\]
for a heap $H$. Then one can show that $\gamma_{[H]}$ belongs to not only $\UU^*$, but also $\II_P^\perp$. Also let
\begin{align}
  \gamma_\mu
    &= \sum_{\ww\in W(\mu)} q^{\inv_P(\ww)} \ww \nonumber \\
    &= \sum_{[H]\in\equivclass{\HH(P,\mu)}} q^{\asc_P([H])} \gamma_{[H]} \in \UU^*_q, \label{eq:gamma_mu=sum_gamma_[H]}  
\end{align}
so $\gamma_\mu$ is automatically in $\II_{P,q}^\perp$.
Note that \eqref{eq:gamma_mu=sum_gamma_[H]} parallels to \eqref{eq:X_P=sum_K_[H]}.
The following theorem enables us to use the noncommutative $P$-Cauchy product $\Omega(\xx,\uu)$ for studying the chromatic quasisymmetric function.
\begin{thm} \label{thm:X_P=Omega}
For a nonnegative integer sequence $\mu=(\mu_1,\dots,\mu_n)$, we have
\begin{equation} \label{eq:K_H=Omega}
  K_{[H]}(\xx) = \langle \Omega(\xx,\uu), \gamma_{[H]} \rangle
\end{equation}
and
\begin{equation} \label{eq:X_P=Omega}
\omega X_P(\xx,q;\mu) = \langle \Omega(\xx,\uu), \gamma_\mu \rangle.
\end{equation}
In particular, $K_{[H]}(\xx)$ is a symmetric function, and so is $\omega X_P(\xx,q;\mu)$.
\end{thm}
\begin{proof}
  Let us first prove \eqref{eq:K_H=Omega}. For \( [H]\in\equivclass{\HH(P,\mu)} \),
  we have
  \begin{align*}
    \left\langle \Omega(\xx,\uu), \gamma_{[H]} \right\rangle
      &= \left\langle \sum_{\ww} F_{d,\Des_P(\ww)}(\xx) u_\ww, \gamma_{[H]} \right\rangle && \mbox{by \eqref{eq:Omega=F},} \\
      &= \sum_{\ww} F_{d,\Des_P(\ww)}(\xx) \sum_{H'\in [H]} \sum_{\vv\in W(H')} \langle u_\ww,\vv \rangle \\
      &= \sum_{H'\in [H]} \sum_{\vv\in W(H')} F_{d,\Des_P(\vv)}(\xx) \\
      &= \sum_{H'\in [H]} K_{H'}(\xx) = K_{[H]}(\xx) && \mbox{by the definition,}
  \end{align*}
  where $\ww$ ranges over all words on the alphabet $[n]$ and $d$ is the length of $\ww$.
  Due to Theorem~\ref{thm:e_k_commute}, we also have
  \[
    H(x,\uu) H(y,\uu) \equiv H(y,\uu) H(x,\uu) \mod \II_P[[x,y]],
  \]
  where \( x \) and \( y \) commute with each other and \( \uu \).
  This implies that the order of the product in the definition of $\Omega(\xx,\uu)$ has no effect when $\Omega(\xx,\uu)$ is considered as an element in $\UU[[\xx]]/\II_P[[\xx]]$. So $K_{[H]}(\xx)$ is a symmetric function because $\gamma_{[H]}\in\II_{P}^\perp$.
  In addition, \eqref{eq:X_P=Omega} follows from \eqref{eq:X_P=sum_K_[H]},
  \eqref{eq:gamma_mu=sum_gamma_[H]} and \eqref{eq:K_H=Omega}.
\end{proof}
Based on this theorem, we will use the following strategy to obtain a
combinatorial description of coefficients of certain bases in the expansion
of \( K_{[H]}(\xx) \) and $\omega X_P(\xx,q;\mu)$.
\begin{cor} \label{cor:strategy}
Suppose that we can write
\begin{equation} \label{eq:Omega=gf}
\Omega(\xx,\uu)\equiv \sum_{\lambda} g_\lambda(\xx) \ff_\lambda(\uu) \mod \II_P[[\xx]]
\end{equation}
for some symmetric function basis $g_\lambda(\xx)$ and noncommutative $P$-symmetric functions $\ff_\lambda(\uu)$.
For a heap \( H \), let \( K_{[H]}(\xx) = \sum_\lambda r_{\lambda,[H]}
g_\lambda(\xx) \). Then we have
\begin{equation} \label{eq:r_lambda,[H]}
  r_{\lambda,[H]} = \langle \ff_\lambda(\uu), \gamma_{[H]} \rangle.
\end{equation}
In particular, if $\omega X_P(\xx,q;\mu) = \sum_{\lambda} r_{\lambda,\mu}(q) g_\lambda(\xx)$, we have
\begin{equation} \label{eq:r_lambda,mu}
  r_{\lambda,\mu}(q) = \langle \ff_\lambda(\uu), \gamma_\mu \rangle.
\end{equation}
\end{cor}
\begin{proof}
The fact that $\gamma_{[H]}$ belongs to $\II_{P}^\perp$ allows us to combine Theorem~\ref{thm:X_P=Omega} with \eqref{eq:Omega=gf}. Then
\[
K_{[H]}(\xx)
  = \left\langle \sum_\lambda g_\lambda(\xx) \ff_\lambda(\uu), \gamma_{[H]} \right\rangle
  = \sum_\lambda \langle \ff_\lambda(\uu), \gamma_{[H]} \rangle g_\lambda(\xx)
\]
and hence we obtain the desired identity.
\end{proof}
In brief, Corollary~\ref{cor:strategy} asserts that assuming the
expression~\eqref{eq:Omega=gf}, a \( \uu \)-monomial expression of
\( \ff_\lambda(\uu) \) provides the \( g \)-expansions of \( K_{[H]}(\xx) \)
and \( \omega X_P(\xx,q;\mu) \).
In particular, if \( \ff_\lambda(\uu) \) has a positive \( \uu \)-monomial
expression, then both \( r_{\lambda,[H]} \) and \( r_{\lambda,\mu}(q) \) are
nonnegative.
The only difference between finding \( r_{\lambda,[H]} \) and finding
\( r_{\lambda,\mu} \) lies in the choice of the word set paired with
\( \ff_{\lambda}(\uu) \); compare \eqref{eq:r_lambda,[H]} and
\eqref{eq:r_lambda,mu}.
Therefore, once we derive a positive \( \uu \)-monomial expression of
\( \ff_\lambda(\uu) \), we will only state an explicit description of the
\( g \)-coefficient of \( \omega X_P(\xx,q;\mu) \) from the \( \uu \)-monomial expression, although the coefficient of \( K_{[H]}(\xx) \) can also be obtained
in the same way.

In what follows, we will introduce several noncommutative $P$-symmetric functions
in order to obtain expansions of $\omega X_P(\xx,q;\mu)$ into certain bases using
Corollary~\ref{cor:strategy}.
More precisely, we will define noncommutative \( P \)-power sum symmetric
functions (Definition~\ref{def:def_pp}), noncommutative \( P \)-Schur functions
(Definition~\ref{def:ss}), and noncommutative \( P \)-monomial symmetric
functions (Definition~\ref{def:mono}).
We will then express $\Omega(\xx,\uu)$ in several ways akin to
\eqref{eq:comm_Cauchy_mh}--\eqref{eq:comm_Cauchy_hm}.
Subsequently, we will derive positive \( \uu \)-monomial expressions for the
noncommutative \( P \)-symmetric functions.
Finally, by applying Corollary~\ref{cor:strategy},
we will present expansions of $\omega X_P(\xx,q;\mu)$ in terms of various bases.

As an example, let us consider the noncommutative $P$-complete homogeneous
symmetric functions $\hh_\lambda(\uu)$, which we already defined.
They provide the expansion of $\omega X_P(\xx,q;\mu)$ in terms of the monomial
symmetric functions, or equivalently the expansion of $X_P(\xx,q;\mu)$ in terms of
the forgotten symmetric functions.
\begin{thm} \label{thm:f-positive}
Let $X_P(\xx,q;\mu) = \sum_\lambda a_\lambda(q) f_\lambda(\xx)$.
Then we have
\[
  a_\lambda(q) = \sum_{\ww} q^{\inv_P(\ww)},
\]
where $\ww$ ranges over all words of type $\mu$ such that when we split $\ww$
from left to right into consecutive segments of lengths
$\lambda_1,\lambda_2,\dots,\lambda_\ell$, each segment has no $P$-descents.
\end{thm}
\begin{proof}
By Corollary~\ref{cor:strategy} and \eqref{eq:Omega=mh}, it suffices to show that
\[
\langle \hh_\lambda(\uu), \gamma_\mu \rangle = \sum_{\ww} q^{\inv_P(\ww)}
\]
where $\ww$ is described above. This immediately follows from \eqref{eq:hh}.
\end{proof}

\begin{exam} \label{ex:f_expansion}
    Consider the running example, Example~\ref{ex:Gamma}.
    To compute the coefficient \( a_{3,1}(q) \), it suffices to find words
    \( \ww=\ww_1 \ww_2 \ww_3 \ww_4 \) such that the subword \( \ww_1 \ww_2 \ww_3 \)
    has no \( P \)-descents. We can find such \framebox{words} as follows:
    \begin{center}
        \includegraphics[scale=0.8]{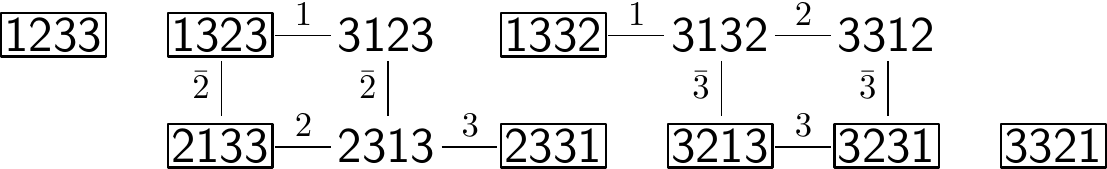}
    \end{center}
    Then we obtain \( a_{3,1}(q) = q^3 + 3 q^2 + 3 q + 1 \).
\end{exam}

\subsection{Noncommutative \texorpdfstring{$P$}{P}-power sum symmetric functions}
When Stanley~\cite{Stanley1995} introduced chromatic symmetric functions $X_G(\xx)$,
he also gave the $p$-expansion of $X_G(\xx)$. For the quasisymmetric case,
Shareshian and Wachs~\cite{SW2016} interpreted the coefficient of $p_n(\xx)$ of
$X_P(\xx,q)$, and conjectured similar interpretations for the coefficients of
$p_\lambda(\xx)$ for arbitrary partitions $\lambda$. Athanasiadis~\cite{Athanasiadis2015}
proved soon thereafter their conjecture, thus the $p$-positivity of
$\omega X_P(\xx,q)$ (equivalently, the $\omega p$-positivity of $X_P(\xx,q)$)
was established.
In his paper, he used some results for $\alpha$-unimodal sequences together
with the result of Shareshian--Wachs.
In this subsection, we define noncommutative \( P \)-power sum symmetric functions,
which give the positive \( p \)-expansion of \( \omega X_P(\xx,q;\mu) \).

As the noncommutative $P$-complete homogeneous symmetric functions $\hh_\lambda(\uu)$, we define noncommutative $P$-power sum symmetric functions using the relation \eqref{eq:relation_peh}.
\begin{defn} \label{def:def_pp}
For $k\ge 1$, the \emph{noncommutative $P$-power sum symmetric function}
$\pp_k(\uu)$ is defined by
\begin{equation} \label{eq:def_pp}
  \pp_k(\uu) = \ee_1(\uu) \hh_{k-1}(\uu) - 2 \ee_2(\uu) \hh_{k-2}(\uu)
                + \cdots + (-1)^{k-1} k \ee_k (\uu).
\end{equation}
For a partition $\lambda=(\lambda_1,\dots,\lambda_\ell)$, define $\pp_\lambda(\uu) = \pp_{\lambda_1}(\uu)\cdots\pp_{\lambda_\ell}(\uu).$
\end{defn}
\begin{prop} \label{prop:Omega=pp}
  We have
  \begin{equation} \label{eq:Omega=pp}
    \Omega(\xx,\uu) \equiv \sum_\lambda \frac{1}{z_\lambda} p_\lambda(\xx) \pp_\lambda(\uu) \mod \II_P[[\xx]].
  \end{equation}
\end{prop}
\begin{proof}
  Although the proof is essentially the same as the proof of
  \cite[Eq. (2.15)]{BF2017}, we include it for the sake of completeness.

  Let $\Sym(\yy)$ be the ring of symmetric polynomials in commuting variables
  $y_1,y_2,\dots$. Then \( \{e_k(\yy)\}_{k\ge 1} \) generates
  $\Sym(\yy)$ and they are algebraically independent.
  Define $\phi:\Sym(\yy)\rightarrow \UU$ by $\phi(e_k(\yy)) = \ee_k(\uu)$ so that
  $\phi(C(\xx,\yy))\equiv \Omega(\xx,\uu)$ modulo $\II_P[[\xx]]$ where
  $C(\xx,\yy)$ is the ordinary Cauchy product. Also we have
  $\phi(p_k(\yy)) \equiv \pp_k(\uu)$ modulo $\II_P$ because $p_k(\yy)$ and
  $\pp_k(\uu)$ obey the same kind of relations \eqref{eq:relation_peh} and
  \eqref{eq:def_pp}, respectively. Hence we deduce \eqref{eq:Omega=pp} from
  \eqref{eq:comm_Cauchy_pp} and \eqref{eq:Omega=mh}.
\end{proof}

Let us now find a \( \uu \)-monomial expression of $\pp_\lambda(\uu)$ in order to
expand \( \omega X_P(\xx,q;\mu) \) in terms of the power sum symmetric functions.
For a word $\ww=\ww_1\cdots\ww_d$, $\ww_i$ is a \emph{left-to-right $P$-maximum} if $\ww_i >_P \ww_j$ for all $1\le j< i$. Of course, $\ww$ always has at least one left-to-right $P$-maximum, namely $\ww_1$.
Then we say that $\ww$ has no nontrivial left-to-right $P$-maxima if for each $2\le i\le d$, there is an integer $j<i$ such that $\ww_i \ngtr_P \ww_j$.

\begin{thm} \label{thm:p-positive}
We have
\begin{equation} \label{eq:pp_monomial}
\pp_k(\uu) \equiv \sum_{\ww} u_\ww \mod \II_P,
\end{equation}
where the sum ranges over all words $\ww$ of length $k$ with no $P$-descents and no nontrivial left-to-right $P$-maxima. Consequently, let $\omega X_P(\xx,q;\mu) = \sum_\lambda \frac{1}{z_\lambda} b_\lambda(q) p_\lambda(\xx)$, then
\[
b_\lambda(q) = \sum_{w\in \NN_\lambda(\mu)} q^{\inv_P(w)}
\]
where $\NN_\lambda(\mu)$ is the set of words $\ww$ of type $\mu$ such that when we split $\ww$ into consecutive subwords of length $\lambda_1,\dots,\lambda_\ell$, each consecutive subword has no $P$-descents and no nontrivial left-to-right $P$-maxima.
\end{thm}
\begin{proof}
Once \eqref{eq:pp_monomial} is verified, the second assertion immediately follows from Corollary~\ref{cor:strategy}, \eqref{eq:Omega=pp} and the definition of $\pp_\lambda(\uu)$. Thus we only prove \eqref{eq:pp_monomial}.

Fix an integer $k>0$. Let $\SSS_r$ be the set of words of length $k$ whose $P$-descent set is $\{1,2,\dots,r-1\}$, i.e., $\SSS_r = \{\ww=\ww_1\cdots\ww_k \mid \ww_1>_P\cdots >_P \ww_r\ngtr_P \ww_{r+1}\ngtr_P\cdots\ngtr_P \ww_k  \}$ for $1\le r \le k$.
By \eqref{eq:def_ee} and \eqref{eq:hh}, we have
\[
\ee_r(\uu) \hh_{k-r}(\uu) = \sum_{ \ww \in \SSS_{r}} u_\ww + \sum_{\ww \in \SSS_{r+1}} u_\ww
\]
for $1\le r < k$.
The definition \eqref{eq:def_pp} of $\pp_k(\uu)$ and the above equation yield that
\[
\pp_k(\uu) = \sum_{r=1}^{k} \sum_{\ww\in\SSS_r} (-1)^{r-1} u_\ww.
\]
Now we construct a sign-reversing involution $\psi$ on $\bigsqcup_r \SSS_r$, which is essentially the same as the one introduced in \cite{SW2016}. Let $\NN_k$ be the set of words of length $k$ with no $P$-descents and no nontrivial left-to-right $P$-maxima. Then $\NN_k\subset\SSS_1\subset \bigsqcup_r \SSS_r$. For $\ww\in\SSS_r$, define $\psi(\ww)$ as follows:
First, if $r=1$ and $\ww\in\NN_k$, then let $\psi(\ww) := \ww$ so that $\ww$ is fixed by $\psi$.
Otherwise, let $i$ be the largest integer such that $\ww_i >_P \ww_j$ for all $r\le j<i$. The fact that $\ww$ is not in $\NN_k$ guarantees the existence of such $i$ and $i>1$.
If $\ww_i >_P \ww_1$, then let $\psi(\ww) := \ww_i \ww_1 \cdots \ww_{i-1}\ww_{i+1}\cdots\ww_k$. In this case, $\psi(\ww)\in\SSS_{r+1}$.
If $\ww_i\ngtr_P\ww_1$, then $\psi(\ww) := \ww_2\cdots\ww_\ell\ww_1\ww_{\ell+1}\cdots\ww_k$ where $\ell$ is the largest integer such that $\ww_1>_P\ww_j$ for all $r\le j\le \ell$, then $\psi(\ww)\in\SSS_{r-1}$. Also one can check that $\psi$ is an involution. Therefore $\psi$ is a sign-reversing involution on $\bigsqcup_r \SSS_r$, and $\NN_k$ is the set of fixed points of $\psi$. To finish the proof, we need that $\psi(\ww)\equiv\ww$ modulo $\II_P$, but it follows from the generators~\eqref{eq:generator_I1} of $\II_P$.
\end{proof}
\begin{exam} \label{ex:p_expansion}
    Let us compute the coefficient \( b_{3,1}(q) \) in our running example.
    Similar to Example~\ref{ex:f_expansion}, it is enough to find words
    \( \ww=\ww_1 \ww_2 \ww_3 \ww_4 \) such that the subword \( \ww_1 \ww_2 \ww_3 \)
    has no \( P \)-descents and no nontrivial left-to-right \( P \)-maxima:
    \begin{center}
        \includegraphics[scale=1]{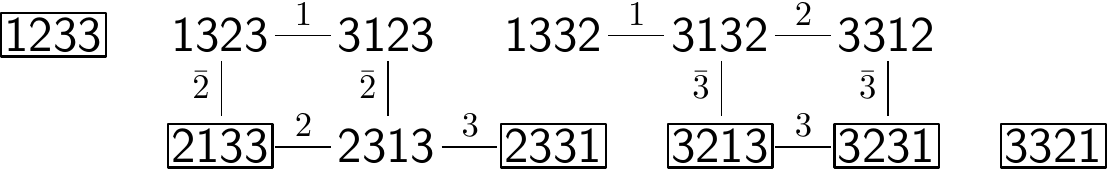}
    \end{center}
    Thus we conclude \( b_{3,1}(q) = q^3 + 2 q^2 + 2 q + 1 \).
\end{exam}

Using the correspondence between words and heaps, one can easily check that
each left-to-right \( P \)-maximum corresponds a sink in the corresponding heap.
Hence we interpret a word with no $P$-descents and no nontrivial left-to-right
$P$-maxima to be a heap with a unique sink.
Using Proposition~\ref{prop:unique_word_and_hh}, we can restate
\eqref{eq:pp_monomial} as follows.
\begin{cor} \label{cor:pp=unique_sink}
We have
\begin{equation} \label{eq:pp=unique_sink}
\pp_k(\uu) \equiv \sum_H u_{\ww_H} \mod \II_P,
\end{equation}
where $H$ ranges over all heaps consisting of $k$ blocks with a unique sink.
\end{cor}

\begin{rmk}
In \cite{GKLLRT1995}, the authors defined two kinds of noncommutative power sum symmetric functions.
The noncommutative $P$-power sum symmetric functions are an analogue
of the noncommutative power sum symmetric functions \emph{of first kind},
and we do not consider an analogue of that of second kind in this paper.
For $p$-positivity of $\omega X_P(\xx,q)$, Alexandersson and
Sulzgruber~\cite[Theorem~7.4]{AS2021} gave a more general result.
They proved that for an arbitrary graph $G$ on $[n]$, $\omega X_G(\xx,q)$ is
$\Psi$-positive where $\Psi$ is the quasisymmetric power sum basis of first kind,
which is the dual basis of the noncommutative power sum basis of first kind.
Moreover, they described the coefficients in $\Psi$ expansion via acyclic
orientations. Our approach extends their result to the (multi-)chromatic
quasisymmetric functions $\omega X_P(\xx,q;\mu)$ of the incomparability graph of an arbitrary poset $P$.
The reason why our noncommutative approach is still valid for this case is that properties of natural unit interval orders are not used in Definition~\ref{def:def_pp} and Theorem~\ref{thm:p-positive}. The only place which relies on properties of $P$ (more precisely, properties of local flips) is in \eqref{eq:Omega=pp}.
But, from the duality between noncommutative power sum symmetric functions and
power sum quasisymmetric functions of first kind, $\Omega(\xx,\uu)$ can be written in terms of noncommutative power sum symmetric functions and the power sum quasisymmetric functions. Therefore one can deduce that $\omega X_P(\xx,q;\mu)$ is $\Psi$-positive, and the coefficients in the $\Psi$-expansion can be described as generating functions of certain words similar to $b_\lambda(q)$ in Theorem~\ref{thm:p-positive}, or equivalently as generating functions of certain heaps.
In fact, when we modify some of our settings, this noncommutative approach gives $\Psi$-positivity of $\omega X_G(\xx,q;\mu)$ for an arbitrary graph $G$. But that modification is outside the scope of this paper, so we omit that.
\end{rmk}

\subsection{Noncommutative \texorpdfstring{$P$}{P}-Schur functions} \label{subsec:Noncomm_Schur}
Schur functions play a central role in algebraic combinatorics.
Since expanding a given symmetric function into Schur functions provides a bridge
between combinatorics and other fields of mathematics,
finding a nice combinatorial model describing the Schur coefficients of the
symmetric function is one of the main questions in the theory of symmetric
functions.
In this subsection, we define noncommutative $P$-Schur functions via the dual
Jacobi--Trudi identity, and describe them in terms of semistandard $P$-tableaux
as the Schur functions. Using this and the duality from \eqref{eq:Omega=ss},
we provide the Schur expansions of $X_P(\xx,q;\mu)$.
The notion of $P$-tableaux was introduced by Gasharov~\cite{Gasharov1996} to prove
$s$-positivity of $X_P(\xx,1;\mu)$. Subsequently, Shareshian and Wachs~\cite{SW2016}
showed $s$-positivity of $X_P(\xx,q)$.

First, we recall the dual Jacobi--Trudi identity which presents a relation between Schur functions $s_\lambda(\xx)$ and elementary symmetric functions $e_\mu(\xx)$:
for a partition \( \lambda \),
\[
s_\lambda(\xx) = \sum_{\sigma\in\mathfrak{S}_m} \sgn(\sigma) e_{\lambda'_1+\sigma(1)-1}(\xx) e_{\lambda'_2+\sigma(2)-2}(\xx)\cdots e_{\lambda'_m+\sigma(m)-m}(\xx)
\]
where $\lambda'$ is the conjugation of $\lambda$ and $m=\lambda_1$.

\begin{defn} \label{def:ss}
For a partition $\lambda$, we define the \emph{noncommutative $P$-Schur function} $\sch_\lambda$ by
\begin{equation} \label{eq:def_ss}
\sch_\lambda(\uu) = \sum_{\sigma\in\mathfrak{S}_m} \sgn(\sigma) \ee_{\lambda'_1+\sigma(1)-1}(\uu) \ee_{\lambda'_2+\sigma(2)-2}(\uu) \cdots \ee_{\lambda'_m+\sigma(m)-m}(\uu),
\end{equation}
where $\lambda'$ is the conjugation of $\lambda$ and $m=\lambda_1$.
\end{defn}
Similar to the Schur functions, we will provide a combinatorial description of
noncommutative $P$-Schur functions. Before describing $\sch_\lambda(\uu)$
combinatorially, we point out the following congruence about the
noncommutative $P$-Cauchy product, and this gives us a duality
between the Schur expansion of $\omega X_P(\xx,q;\mu)$ and $\sch_\lambda(\uu)$.
The proof is similar to the proof of Proposition~\ref{prop:Omega=pp},
so we omit it.
\begin{prop} \label{prop:Omega=ss}
  We have
  \begin{equation} \label{eq:Omega=ss}
    \Omega(\xx,\uu) \equiv \sum_\lambda s_\lambda(\xx) \sch_\lambda(\uu) \mod \II_P[[\xx]].
  \end{equation}
\end{prop}

\begin{defn}
For a partition $\lambda$, a \emph{semistandard $P$-tableau of shape $\lambda$} is
a filling of the Young diagram of shape $\lambda$ with $[n]$ satisfying that
\begin{enumerate}[label=(\roman*)]
\item each row is non-$P$-decreasing from left to right, and
\item each column is $P$-increasing from top to bottom.
\end{enumerate}
A semistandard $P$-tableau $T$ is of \emph{type $\mu=(\mu_1,\dots,\mu_n)$} if each $i\in [n]$ appears $\mu_i$ times in $T$. We denote the set of all semistandard $P$-tableaux of shape $\lambda$ by $\TT_P(\lambda)$. The \emph{reading word} $\ww(T)$ of $T$ is the word obtained by reading $T$ from bottom to top, beginning with the leftmost column of $T$ and working from left to right. 
\end{defn}
\begin{exam}
Let $P=P(2,4,4,5,5)$ and $\lambda = (4,2,1)$. In Figure~\ref{fig:tableaux},
the first one is a semistandard $P$-tableau of type $(1,1,1,1,3)$ and the reading
word is $\mathsf{5315254}$.
The others are not semistandard $P$-tableaux because the first column of the second
tableau is not $P$-increasing and the first row of the last tableau is not
non-$P$-decreasing.
\begin{figure}
\centering
\ytableaushort{1254,35,5} \qquad
\ytableaushort{1254,45,5} \qquad
\ytableaushort{1253,35,5}
\caption{}
\label{fig:tableaux}
\end{figure}
\end{exam}
\begin{rmk}
  Gasharov~\cite{Gasharov1996} and Shareshian--Wachs~\cite{SW2016} defined
  $P$-tableaux slightly differently.
  Their definition is the conjugate version of ours. We use the above definition
  which is parallel to the definition of semistandard Young tableaux.
\end{rmk}

\begin{thm} \label{thm:s-positive}
We have
\begin{equation} \label{eq:ss_P-tabluea}
\sch_\lambda(\uu) \equiv \sum_{T\in\TT_P(\lambda)} u_{\ww(T)} \mod \II_P.
\end{equation}
Consequently, $\omega X_P(\xx,q;\mu)$ is $s$-positive and its coefficient of $s_\lambda$ counts semistandard $P$-tableaux of shape $\lambda$ and of type $\mu$. In other words,
\[
\omega X_P(\xx,q;\mu) = \sum_{T} q^{\inv_P(T)} s_{\sh(T)}(\xx),
\]
where $T$ ranges over all semistandard $P$-tableaux of type $\mu$ and $\sh(T)$ denotes the shape of $T$.
\end{thm}
\begin{proof}
As Theorem~\ref{thm:p-positive}, the second part of this theorem follows immediately from \eqref{eq:Omega=ss} and \eqref{eq:ss_P-tabluea}.
To show \eqref{eq:ss_P-tabluea}, we use an idea of the
Lindstr\"om-Gessel-Viennot lemma, which is a useful tool for enumerating
non-intersecting paths (see \cite[Theorem~7.16.1]{EC2} for more details).

Let $P=P(\mm)$ where $\mm=(m_1,\dots,m_n)$.
Then we construct a edge-weighted graph with vertex set $\mathbb{Z}\times [n+1]$.
Two distinct vertices $(i, a)$ and $(j, b)$ with $a<b$ are adjacent by an edge $e$ if either
\begin{enumerate}[label=(\roman*)]
\item $i=j$ and $b = a+1$ (in this case, we assign $\wt(e)=1$), or
\item $i = j-1$ and $b = m_a + 1$ (in this case, we assign $\wt(e)=u_a$).
\end{enumerate}
For any $i,k\ge 0$, let $A=(i+k, n+1)$ and $B=(i,1)$, and let $\pathp$ be a path from $A$ to $B$. We define $\wt(\pathp) = \wt(e_1)\cdots\wt(e_r)$ where $\pathp = (A=p_0\overset{e_1}{\longrightarrow}p_1\overset{e_2}{\longrightarrow}\dots\overset{e_r}{\longrightarrow}p_r=B)$. Then by construction we have
\[
\ee_k(\uu) = \sum_\pathp \wt(\pathp)
\]
where the sum is over all paths $\pathp$ from $A$ to $B$. See Figure~\ref{fig:all_paths}.
For an $m$-tuple $(\pathp_1,\dots,\pathp_m)$ of paths, define $\wt(\pathp_1,\dots,\pathp_m) = \wt(\pathp_1)\cdots\wt(\pathp_m)$.
\begin{figure}
\centering
\includegraphics[width=0.18\linewidth]{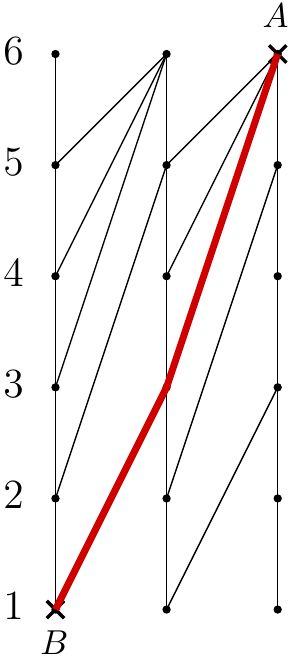}\qquad\quad
\includegraphics[width=0.18\linewidth]{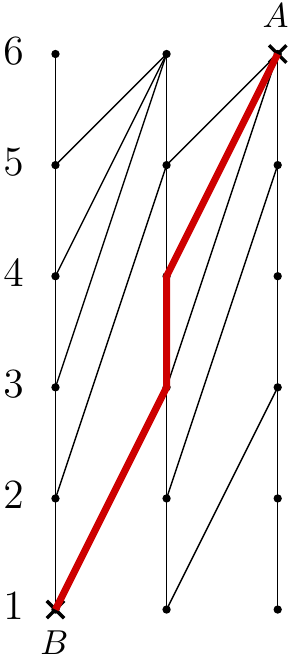}\qquad\quad
\includegraphics[width=0.18\linewidth]{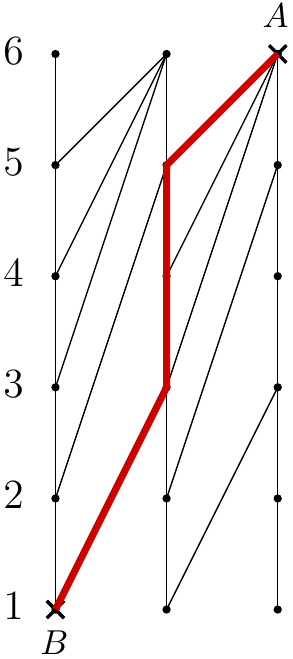}\qquad\quad
\includegraphics[width=0.18\linewidth]{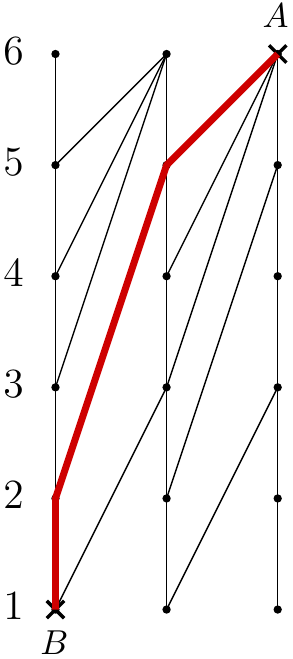}
\caption{For $P=P(2,4,5,5,5)$ and $k=2$, we illustrate all paths from $A$ to $B$. From left to right, the weights of paths are $u_3u_1$, $u_4u_1$, $u_5u_1$, and $u_5u_2$, respectively. Their summation coincides with $\ee_2(\uu) = u_3u_1+u_4u_1+u_5u_1+u_5u_2$.}
\label{fig:all_paths}
\end{figure}

For a partition $\lambda$, let $m=\lambda_1$. Also let
$A_i = (m-i+\lambda'_i, n+1)$ and $B_i=(m-i, 1)$ for $1\le i\le m$.
Then $\sch_\lambda(\uu)$ is equal to the signed sum of weights of $m$-tuples
$(\pathp_1,\dots,\pathp_m)$ of paths from $(A_1,\dots,A_m)$ to
$(B_{\sigma(1)},\dots,B_{\sigma(m)})$ over all $\sigma\in\mathfrak{S}_m$. 
Since our weights on edges (and paths) do not commute, we cannot apply the 
Lindstr\"om--Gessel--Viennot lemma directly. Hence we modify the
Lindstr\"om--Gessel--Viennot lemma, which will work on our case modulo $\II_P$.

We construct a sign-reversing involution on the $m$-tuples of paths, whose fixed
points are $m$-tuples of non-intersecting paths from $(A_1,\dots,A_m)$ to
$(B_1,\dots,B_m)$. Suppose that $(\pathp_1,\dots,\pathp_m)$ be an $m$-tuple of
paths such that there are some intersecting points.
Choose the highest intersecting point $p$ (if there are more than one such points,
we choose the leftmost one among them.).
Then only two paths intersect at $p$, and their starting positions are next to each other, so we denote them by $A_k$ and $A_{k+1}$. We can split $\pathp_k$ into two paths $\pathp_k^{(1)}$ and $\pathp_k^{(2)}$ such that $\pathp_k^{(1)}$ (respectively, $\pathp_k^{(2)}$) is the path from $A_k$ to $p$ (respectively, from $p$ to $B_{\sigma(k)}$ for some $\sigma\in\mathfrak{S}_m$).
We write $\pathp_k = (A_k\overset{\pathp_k^{(1)}}{\longrightarrow} p \overset{\pathp_k^{(2)}}{\longrightarrow} B_{\sigma(k)})$.
Similarly, we split $\pathp_{k+1}$ as $\pathp_{k+1} = (A_{k+1}\overset{\pathp_{k+1}^{(1)}}{\longrightarrow} p \overset{\pathp_{k+1}^{(2)}}{\longrightarrow} B_{\sigma(k+1)})$.
Also let $\wt(\pathp_k^{(2)})=u_{\ww}$ and $\wt(\pathp_{k+1}^{(2)})=u_{\vv}$ for some words $\ww\in\EE_{r}$ and $\vv\in\EE_{s}$ where $\EE_\ell$ is given in the proof of Theorem~\ref{thm:e_k_commute}. Let $(\vv',\ww')=\psi_{r, s}(\ww, \vv)$, and then there is the path $\pathp'_k$ (respectively, $\pathp'_{k+1}$) from $p$ to $B_{\sigma(k)}$ (respectively, $B_{\sigma(k+1)}$) whose weight is $u_{\ww'}$ (respectively, $u_{\vv'}$).
We now define a map as follows:
\[
(\pathp_1,\dots,\pathp_m)\longmapsto
\begin{cases}
~ (\pathp_1,\dots,\pathp_m) & \begin{array}{l}
    \mbox{if $(\pathp_1,\dots,\pathp_m)$ have} \\
    \quad\mbox{no intersecting point,}
  \end{array} \\
~ (\pathp_1,\dots,\widetilde{\pathp}_k, \widetilde{\pathp}_{k+1},\dots,\pathp_m) & ~\mbox{ otherwise}, 
\end{cases}
\]
where $\widetilde{\pathp}_k=(A_k\overset{\pathp_k^{(1)}}{\longrightarrow} p \overset{\pathp'_{k+1}}{\longrightarrow} B_{\sigma(k+1)})$ and $\widetilde{\pathp}_{k+1}=(A_{k+1}\overset{\pathp_{k+1}^{(1)}}{\longrightarrow} p \overset{\pathp'_k}{\longrightarrow} B_{\sigma(k)})$. Obviously, this map is a sign-reversing involution. Furthermore it preserves weights modulo $\II_P$ because
\begin{align*}
\wt(\pathp_k)\wt(\pathp_{k+1})
&= \wt(\pathp_k^{(1)})\cdot u_\ww\cdot \wt(\pathp_{k+1}^{(1)})\cdot u_\vv \\
&\equiv \wt(\pathp_k^{(1)})\cdot \wt(\pathp_{k+1}^{(1)})\cdot u_\ww\cdot u_\vv \mod \II_P \\
&\equiv \wt(\pathp_k^{(1)})\cdot \wt(\pathp_{k+1}^{(1)})\cdot u_{\vv'}\cdot u_{\ww'} \mod \II_P & \mbox{by the property of $\psi_{r,s},$} \\
&\equiv \wt(\pathp_k^{(1)})\cdot u_{\vv'}\cdot \wt(\pathp_{k+1}^{(1)})\cdot  u_{\ww'} \mod \II_P \\
&= \wt(\widetilde{\pathp}_k)\wt(\widetilde{\pathp}_{k+1}).
\end{align*}

Therefore, with the above involution, the Lindstr\"om-Gessel-Viennot lemma yields
\[
  \sch_\lambda(\uu)\equiv \sum_{(\pathp_1,\dots,\pathp_m)}
    \wt(\pathp_1,\dots,\pathp_m) \mod \II_P,
\]
where $(\pathp_1,\dots,\pathp_m)$ ranges over all non-intersecting paths from
$(A_1,\dots,A_m)$ to $(B_1,\dots,B_m)$. One can easily construct a bijection
$\phi$ between such paths and semistandard $P$-tableaux of shape $\lambda$
satisfying $\wt(\pathp_1,\dots,\pathp_m)=u_{\ww(T)}$ where
$T=\phi(\pathp_1,\dots,\pathp_m)$. This completes the proof.
\end{proof}
\begin{exam}
    Returning to the running example, to obtain the coefficient \( a_{3,1}(q) \)
    of \( s_{3,1}(\xx) \) of \( \omega X_P(\xx,q;\mu) \), it suffices to find words
    \( \ww=\ww_1 \ww_2 \ww_3 \ww_4 \) such that \( \ww \) is a reading word
    for some semistandard \( P \)-tableaux of shape \( (3,1) \):
    \begin{center}
        \includegraphics[scale=1]{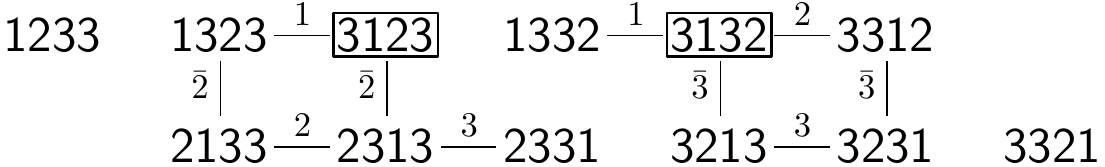}
    \end{center}
    Therefore we have \( a_{3,1}(q) = q^2 + q \).
\end{exam}
\begin{rmk}
    In \cite{KP2021}, the authors defined an equivalence relation on words,
    which differs from ours.
    In addition, they conjectured that the Schur expansions of equivalence classes
    are given by reading words of \( P \)-tableaux that appear in the equivalence
    class; see \cite[Conjecture~4.16]{KP2021}.
    Their equivalence classes are finer than ours, so Theorem~\ref{thm:s-positive}
    gives a weak answer for the conjecture.
\end{rmk}

\subsection{Noncommutative \texorpdfstring{$P$}{P}-monomial symmetric functions} \label{sec:Noncomm_mono}
We devote the remainder of this section to establish noncommutative $P$-monomial
symmetric functions.

Finding a combinatorial interpretation of the $e$-coefficients of chromatic
quasisymmetric functions is a famous long-standing open problem in algebraic
combinatorics. Thanks to the duality from \eqref{eq:Omega=hm} which will be
described, monomials appearing in the noncommutative $P$-monomial symmetric
function $\mono_\lambda(\uu)$ modulo $\II_P$ gives us the coefficient of
$e_\lambda(\xx)$ of $X_P(\xx,q;\mu)$ (equivalently, the coefficient of
$h_\lambda(\xx)$ of $\omega X_P(\xx,q;\mu)$).

Before defining the noncommutative $P$-monomial symmetric functions, consider the monomial symmetric functions $m_\lambda(\xx)$. The transition matrix between $\{m_\lambda(\xx)\}_{\lambda\in\Par}$ and $\{e_\lambda(\xx)\}_{\lambda\in\Par}$ has a simple combinatorial interpretation;
see \cite[Proposition~7.4.1 and Theorem~7.4.4]{EC2}.
\begin{thm} \label{thm:e=Mm_ordianry}
Let $e_\lambda(\xx) = \sum_\mu M_{\lambda,\mu} m_\mu(\xx)$. Then $M_{\lambda,\mu}$ is the number of $(0,1)$-matrices whose $i$th row sum equals $\lambda_i$, and $j$th column sum equals $\mu_j$ for all $i,j$. In particular,
\[
M_{\lambda,\mu} = 0 \mbox{ unless } \mu \trianglelefteq \lambda', \qand M_{\lambda, \lambda'} = 1.
\]
\end{thm}
From this theorem, the matrix $(M_{\lambda, \mu'})_{\lambda,\mu}$ with reverse lexicographic order is a lower triangular matrix with 1's on its diagonal. Then we obtain immediately the following proposition.
\begin{prop} \label{prop:m=Ne}
Let $m_\lambda(\xx) = \sum_{\mu} N_{\lambda,\mu} e_\mu(\xx)$. Then $N_{\lambda,\mu} = 0$ unless $\mu' \trianglelefteq \lambda$, and $N_{\lambda,\lambda'} = 1$.
\end{prop}

\begin{defn} \label{def:mono}
For a partition $\lambda$, the \emph{noncommutative $P$-monomial symmetric function} $\mono_\lambda(\uu)$ is defined by
\begin{equation} \label{eq:def_mono}
\mono_\lambda(\uu) = \sum_{\mu} N_{\lambda,\mu} \ee_\mu(\uu),
\end{equation}
where $N_{\lambda,\mu}$ is given in Proposition~\ref{prop:m=Ne}.
\end{defn}
As \eqref{eq:comm_Cauchy_hm}, we can write $\Omega(\xx,\uu)$ via complete homogeneous symmetric functions $h_\lambda(\xx)$ and noncommutative $P$-monomial symmetric functions $\mono_\lambda(\uu)$.
\begin{prop} \label{prop:Omega=hm}
  We have
  \begin{equation} \label{eq:Omega=hm}
    \Omega(\xx,\uu) \equiv \sum_{\lambda} h_\lambda(\xx) \mono_\lambda(\uu) \mod \II_P[[\xx]].
  \end{equation}
\end{prop}

Let us first consider a general case. In his seminal paper~\cite{Stanley1995},
Stanley showed that $e$-coefficients of chromatic symmetric functions are
related to acyclic orientations of the graph, and Shareshian and Wachs~\cite{SW2016}
proved a refined result for chromatic quasisymmetric functions.
\begin{thm}[\cite{SW2016,Stanley1995}] \label{thm:c_lambda=acyclic}
For a natural unit interval order $P$, let $X_P(\xx,q) = \sum_\lambda c_\lambda(q) e_\lambda(\xx)$. Then for $k\ge 1$,
\[
\sum_{\ell(\lambda)=k} c_\lambda(q) = \sum_{\oo\in AO(P,k)} q^{\asc_P(\oo)},
\]
where $AO(P,k)$ is the set of all acyclic orientations of $P$ with $k$ sinks, and $\asc_P(\oo)$ is the number of ascent edges in $\oo$.
\end{thm}
Although Stanley proved this theorem for arbitrary graphs, the chromatic
quasisymmetric functions are in general not symmetric functions, so Shareshian
and Wachs showed this for natural unit interval orders.

Theorem~\ref{thm:c_lambda=acyclic} is an evidence for the conjectures of Stanley--Stembridge and Shareshian--Wachs which say that $c_\lambda(q)$ is a polynomial with nonnegative coefficients, i.e., $c_\lambda(q)\in\mathbb{N}[q]$.
Furthermore, this theorem suggests that $c_\lambda(q)$ can be stated as a sum of certain acyclic orientations of $P$ with some conditions.
For $X_P(\xx,q;\mu)$, we have a similar description of $\sum_{\ell(\lambda)=k} c_\lambda(q)$ in terms of certain heaps. To show this, we need the following lemma which is an easy exercise.
\begin{lem} \label{lem:sum_m=sum_eh}
We have
\[
\sum_{\substack{\lambda\vdash d \\ \ell(\lambda)=k}} m_\lambda(\xx) = \sum_{j=k}^{d} (-1)^{j-k} \binom{j}{k} e_j(\xx) h_{d-j}(\xx).
\]
\end{lem}

\begin{thm} \label{thm:sum_m=sink_heap}
Let $P$ be a natural unit interval order. Then for \( d, k\ge 1 \), we have
\begin{equation}
\sum_{\substack{\lambda\vdash d \\ \ell(\lambda)=k}} \mono_\lambda(\uu) \equiv \sum_H u_{\ww_H} \mod \II_P \nonumber,
\end{equation}
where $H$ ranges over all heaps of $P$ consisting of $d$ blocks with $k$ sinks.
Consequently, let $X_P(\xx,q;\mu) = \sum_\lambda c_\lambda(q) e_\lambda(\xx)$.
Then we have
\[
  \sum_{\ell(\lambda)=k} c_\lambda(q) = \sum_H q^{\asc_P(H)},
\]
where $H$ ranges over all heaps of type $\mu$ with $k$ sinks.
\end{thm}
\begin{proof}
From Lemma~\ref{lem:sum_m=sum_eh}, we have
\begin{equation} \label{eq:sum_m=sum_eh_noncomm}
\sum_{\substack{\lambda\vdash d \\ \ell(\lambda)=k}} \mono_\lambda(\uu) \equiv \sum_{j=k}^{d} (-1)^{j-k} \binom{j}{k} \ee_j(\uu) \hh_{d-j}(\uu) \mod \II_P.
\end{equation}
Let us write $\ee_j(\uu) \hh_{d-j}(\uu)$ as a sum of monomials:
\[
\ee_j(\uu) \hh_{d-j}(\uu) = \sum_{\substack{\ww_1>_P\dots >_P \ww_j \\ \ww_{j+1}\ngtr_P\dots \ngtr_P \ww_d}} u_\ww.
\]
Proposition~\ref{prop:algebraic_expression_local_flip} says that if $\ww$ and $\ww'$ correspond to the same heap $H$, i.e., $\ww,\ww' \in W(H)$, then $u_\ww \equiv u_{\ww'} \mod \II_P$.
Let $\ww$ be a word satisfying the condition in the sum on the right hand side. Then the heap corresponding to $\ww$ has at least $j$ sinks.
Conversely, let $H$ be a heap with $\ell$ sinks. Then there are $\binom{\ell}{j}$ words in $W(H)$ satisfying the condition. Indeed, first choose $j$ sinks, and write them in $P$-decreasing order. Delete the chosen sinks from $H$, denoting the resulting heap by $H'$, and then by Proposition~\ref{prop:unique_word_and_hh} there is only one word $\ww_{H'}$ corresponding to $H'$ with no $P$-descents, so let $\ww_{j+1}\cdots\ww_d = \ww_{H'}$.
Then by Proposition~\ref{prop:algebraic_expression_local_flip}, we can rewrite $\ee_j(\uu) \hh_{d-j}(\uu)$ as
\[
\ee_j(\uu) \hh_{d-j}(\uu) \equiv \sum_{s(H)\ge j} \binom{s(H)}{j} u_{\ww_H} \mod \II_P,
\]
where $s(H)$ is the number of sinks of $H$. Combining with \eqref{eq:sum_m=sum_eh_noncomm} yields
\begin{align*}
\sum_{\substack{\lambda\vdash d \\ \ell(\lambda)=k}} \mono_\lambda(\uu) &\equiv \sum_{j=k}^{d} (-1)^{j-k} \binom{j}{k} \sum_{s(H)\ge j} \binom{s(H)}{j} u_{\ww_H}  \mod \II_P \\
&\equiv \sum_{s(H)\ge k} \sum_{j=k}^{s(H)} (-1)^{j-k} \binom{s(H)}{j} \binom{j}{k} u_{\ww_H} \mod \II_P \\
&\equiv \sum_{s(H)= k} u_{\ww_H} \mod \II_P,
\end{align*}
since $\sum_{j=k}^{\ell} (-1)^{j-k} \binom{\ell}{j} \binom{j}{k} = \delta_{\ell,k}$.
\end{proof}
We remark that the case \( k=1 \) in the theorem is compatible with
Corollary~\ref{cor:pp=unique_sink} since \( p_k(\xx)=m_k(\xx) \) and thus
\( \pp_k(\uu) \equiv \mono_k(\uu) \) modulo \( \II_P \).

According to Theorem~\ref{thm:sum_m=sink_heap}, one can try to describe
the coefficient $c_\lambda(q)$ as a sum of certain heaps with $\ell(\lambda)$ sinks
and some extra conditions. In this sense, we provide a combinatorial description of
$c_\lambda(q)$ where $\lambda$ is of two-column shape or of hook shape.

\subsubsection{Indexed by two-column shapes}
We first check the following equation for the monomial symmetric functions indexed
by partitions of two-column shape: for $k\ge \ell\ge 0$,
\[
e_{(k,\ell)}(\xx) = m_{(2^\ell, 1^{k-\ell})}(\xx) + \sum_{i=0}^{\ell-1} \binom{k+\ell-2i}{\ell-i} m_{(2^i,1^{k+\ell-2i})}(\xx).
\]
This equation follows immediately from Theorem~\ref{thm:e=Mm_ordianry}.
We thus obtain a noncommutative $P$-analogue of this equation module \( \II_P \):
\begin{equation} \label{eq:ee_kl=mono_two_column}
\ee_{(k,\ell)}(\uu) \equiv \mono_{(2^\ell, 1^{k-\ell})}(\uu) + \sum_{i=0}^{\ell-1} \binom{k+\ell-2i}{\ell-i} \mono_{(2^i,1^{k+\ell-2i})}(\uu) \mod \II_P.
\end{equation}
We will find conditions for heaps which contribute to $\mono_\lambda(\uu)$ where $\lambda$ is of two-column shape using this congruence.

Given a heap $H$, the \emph{rank} of a block $p$ is the height of $p$ in the diagram of $H$, denoted by $\rank(p)$. In other words, $\rank(p)$ is one more than the length of a longest path from $p$ to some sink in $H$. In particular, $\rank(p)=1$ if and only if $p$ is a sink. Also we define the \emph{rank} of $H$ by the maximum rank of blocks.

The following argument is similar to the proof of Theorem~\ref{thm:e_k_commute}.
Consider $\ee_{(k,\ell)}(\uu)$ for $k\ge \ell\ge 0$. Let
\[
  \EE_{k,\ell} = \{ \ww_1\cdots\ww_{k+\ell} \mid \ww_1>_P\dots>_P\ww_k \mbox{ and } \ww_{k+1}>_P\dots>_P\ww_{k+\ell} \}.
\]
Then, by definition,
\[
\ee_{(k,\ell)}(\uu) = \sum_{\ww\in\EE_{k,\ell}} u_\ww.
\]
For a word $\ww$, we denote by $H_\ww$ the heap corresponding to $\ww$.
Note that for each $\ww\in\EE_{k,\ell}$, $H_\ww$ is of rank at most $2$.
We can classify connected heaps of rank at most $2$ as follows.
For a connected heap $H$ at most rank $2$, let \( n_1 \) and \( n_2 \) be the
numbers of blocks of rank 1 and 2, respectively.
Note that \( |n_1-n_2| \le 1 \). We say that the heap \( H \) is
\begin{enumerate}[label=(\roman*)]
\item \emph{of type S} if \( n_1=1 \) and \( n_2=0 \) (that is, the rank of \( H \) is 1);
\item \emph{of type N} if \( n_1 = n_2 \);
\item \emph{of type M} if \( n_1 = n_2+1\ge 2 \);
\item \emph{of type W} if \( n_1 = n_2-1 \).
\end{enumerate}
An example of each type is depicted in Figure~\ref{fig:type_SNMW}.
\begin{figure}
\centering
\begin{tikzpicture}[scale=.4]
  \draw (0,0) rectangle ++(\RECTX,\RECTY);
  \draw (3.5,0) rectangle ++(\RECTX,\RECTY);
  \draw (5,1) rectangle ++(\RECTX,\RECTY);
  \draw (6.5,0) rectangle ++(\RECTX,\RECTY);
  \draw (8,1) rectangle ++(\RECTX,\RECTY);
  \draw (11.5,0) rectangle ++(\RECTX,\RECTY);
  \draw (13,1) rectangle ++(\RECTX,\RECTY);
  \draw (14.5,0) rectangle ++(\RECTX,\RECTY);
  \draw (16,1) rectangle ++(\RECTX,\RECTY);
  \draw (17.5,0) rectangle ++(\RECTX,\RECTY);
  \draw (21,1) rectangle ++(\RECTX,\RECTY);
  \draw (22.5,0) rectangle ++(\RECTX,\RECTY);
  \draw (24,1) rectangle ++(\RECTX,\RECTY);
  \draw (25.5,0) rectangle ++(\RECTX,\RECTY);
  \draw (27,1) rectangle ++(\RECTX,\RECTY);
  \draw [thick] (-1,0) -- (30.5,0);
\end{tikzpicture}
\caption{From left to right, each connected component is of type S, N, M, and W, respectively.} \label{fig:type_SNMW}
\end{figure}
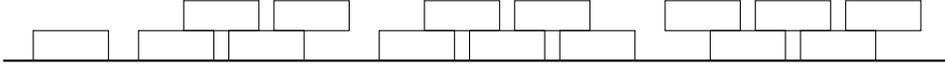

One can easily check that a connected component of type M can be flipped
to that of type W via a series of local flips, and vice versa.

\begin{thm} \label{thm:m_two_column}
For $k\ge \ell\ge0$, we have
\begin{equation} \label{eq:mono_two-column}
\mono_{(2^\ell,1^{k-\ell})}(\uu) \equiv \sum_H u_{\ww_H} \mod \II_P,
\end{equation}
where $H$ ranges over all heaps of $P$ such that $H$ consists of $k$ blocks of rank 1 and $\ell$ blocks of rank 2, and has no connected component of type W. Consequently, let $X_P(\xx,q;\mu) = \sum_\lambda c_\lambda(q) e_\lambda(\xx)$, then
\[
c_{(2^\ell,1^{k-\ell})}(q) = \sum_H q^{\asc_P(H)},
\]
where $H$ ranges over such heaps of type $\mu$.
\end{thm}
\begin{proof}
We use induction on $\ell$, the base case $\ell=0$ being trivial.

Since two distinct words $\ww,\ww'\in\EE_{k,\ell}$ may correspond to the same heap,
we first enumerate words in $\EE_{k,\ell}$ corresponding to $H_\ww$
for a given word $\ww\in\EE_{k,\ell}$.
Let $n_1$ and $n_2$ be the number of blocks of $H_\ww$ of rank 1 and rank 2,
respectively. For $1\le i\le k$, each block corresponding to $\ww_i$ is of rank 1,
and every block of rank 2 corresponds to $\ww_j$ for some $k+1 \le j\le k+\ell$.
Then $n_1\ge k$ and $n_2\le \ell$.
Furthermore, for $k+1 \le j\le k+\ell$, if a block corresponding to $\ww_j$ is of
rank 1, then it forms a connected component of type S.
Let $n_S$ be the number of connected components of $H_\ww$ of type S.
Therefore we have
\begin{equation} \label{eq:numbers_words=H_w}
  \#\{\ww'\in\EE_{k,\ell} \mid H_{\ww'} = H_\ww \} = \binom{n_S}{\ell-n_2}.
\end{equation}

For a word $\ww\in\EE_{k,\ell}$, let $\WM(H_\ww)$ be the heap obtained from $H_\ww$
by flipping all connected components of type W to ones of type M via a series of
local flips.
By Proposition~\ref{prop:algebraic_expression_local_flip}, $u_\ww \equiv u_{\ww_{\WM(H_\ww)}} \mod \II_P$.
Define the multiset $\WM(\EE_{k,\ell})$ by
\[
\WM(\EE_{k,\ell}) = \{\WM(H_\ww) \mid \ww\in\EE_{k,\ell} \}.
\]
Then we can write
\begin{equation} \label{eq:e=sum_WM}
\ee_{(k,\ell)}(\uu) \equiv \sum_{H\in\WM(\EE_{k,\ell})} u_{\ww_H} \mod \II_P.
\end{equation}
For $H\in\WM(\EE_{k,\ell})$, let us now find the multiplicity of $H$.
Again let $n_2$ be the number of blocks of $H$ of rank 2, and $n_S, n_M$ be
the number of connected components of $H$ of type S and type M, respectively.
Since $H$ has no connected component of type $W$, one can quickly check that
$n_M + n_S = k+\ell-2 n_2$.
In addition, we need the following observation.
Choose $j$ connected components of type M,
and flip them to ones of type W. Then the resulting heap $H'$ has $n_2+j$
blocks of rank 2, and $n_S$ connected components of type S. Of course,
$\WM(H') = H$. Hence by this observation and \eqref{eq:numbers_words=H_w},
we obtain that
\begin{align*}
\mbox{the multiplicity of $H$ in $\WM(\EE_{k,\ell})$} &= \sum_{j=0}^{n_M} \binom{n_M}{j} \binom{n_S}{\ell-n_2-j} \\
&= \binom{n_M + n_S}{\ell-n_2} \\
&= \binom{k+\ell-2n_2}{\ell-n_2}.
\end{align*}
Apply this to \eqref{eq:e=sum_WM}, and hence by induction on $\ell$, comparing with
\eqref{eq:ee_kl=mono_two_column} completes the proof.
\end{proof}

\begin{exam}
Let $P=P(2,3,4,5,5)$ and $\mu=(1^5)$. See Figure~\ref{fig:heap_orbit_2345}.
There are two heaps of type $\mu$ of rank 2; they are connected, and one is of
type M while the other is of type W.
Then by Theorem~\ref{thm:m_two_column}, we have $c_{2,2,1}(q) = q^2$ since the
ascent number of the heap of type M is 2.
\end{exam}

From the interpretation of the coefficients \( c_\lambda(q) \) in
Theorem~\ref{thm:m_two_column}, we obtain somewhat more information for the
coefficients. To simplify, we consider the chromatic symmetric function
\( X_P(\xx,q) = X_P(\xx,q;1^n) \). If there is a triangle in the graph \( P \),
every heap of \( P \) has rank at least 3. Therefore, in this case, the theorem
implies that \( c_\lambda(q) = 0 \) for all partitions \( \lambda \) of 2-column
shape. Suppose that \( P \) has no triangles. Since \( P \) is the incomparability
graph of a natural unit interval order, \( P \) has to be a disjoint union of paths.
Let \( H \) be a heap satisfying the conditions in Theorem~\ref{thm:m_two_column}.
Then the rank condition and the type condition force that each connected component
of \( H \) consisting of odd blocks forms of type \( M \).
On the other hand, each connected component consisting of even blocks can be form
of the following two cases:
\begin{center} \vspace{0.15cm}
  \begin{tikzpicture}[scale=.5]
    \draw (0,1) rectangle ++(\RECTX,\RECTY);
    \draw (1.5,0) rectangle ++(\RECTX,\RECTY);
    \draw (3,1) rectangle ++(\RECTX,\RECTY);
    \draw (4.5,0) rectangle ++(\RECTX,\RECTY);
    \draw [thick] (-1,0) -- (8,0);
  \end{tikzpicture} \qquad
  \begin{tikzpicture}[scale=.5]
    \draw (0,0) rectangle ++(\RECTX,\RECTY);
    \draw (1.5,1) rectangle ++(\RECTX,\RECTY);
    \draw (3,0) rectangle ++(\RECTX,\RECTY);
    \draw (4.5,1) rectangle ++(\RECTX,\RECTY);
    \draw [thick] (-1,0) -- (8,0);
  \end{tikzpicture}
\end{center}
We summarize the above discussion.
\begin{cor} \label{cor:2_column_unimodal}
  Let \( P \) be the incomparability graph of a natural unit interval order
  with \( n \) vertices, and \( X_P(\xx,q) = \sum_\lambda c_\lambda(q)
  e_\lambda(\xx) \). Let \( n_e \) and \( n_o \) be the numbers of connected
  components of \( P \) consisting of even and odd vertices, respectively.
  Then for a partition \( \lambda \) of 2-column shape, we have
  \[
    c_\lambda(q) = 
    \begin{cases}
      q^{(n-2n_e-n_o)/2} (1+q)^{n_e} & \mbox{if \( P \) is triangle-free and \( \lambda'=((n+n_o)/2, (n-n_o)/2) \),} \\ 
      0 & \mbox{otherwise.}
    \end{cases}
  \]
  In particular, the coefficients \( c_\lambda(q) \) where \( \lambda \) are of
  2-column shape are unimodal.
\end{cor}

\subsubsection{Indexed by hook shapes}
In this subsection, we deal with the coefficient \( c_\lambda(q) \) where
\( \lambda \) is a partition of hook shape. We first define some notions.
For $\{a_1<\dots<a_k\}\subset [n]$, we say that $\{a_1, \dots, a_k\}$
\emph{forms a $P$-path} if the subgraph of $\inc(P)$ induced by
$\{a_1,\dots,a_k\}$ is a path. In terms of the poset structure of $P$,
$\{a_1,\dots,a_k\}$ forms a $P$-path if and only if
$a_1 \nless_P a_2 \nless_P \dots \nless_P a_k$ and $a_i <_P a_j$ if $j-i>1$.

Now consider the following situation. Let $H_1$ be a heap with the unique sink $p_1$, and $(p_2,p_1,p_0)$ be a flippable triple in $H_1$ as
\begin{center} \vspace{0.15cm}
\includegraphics[scale=0.75]{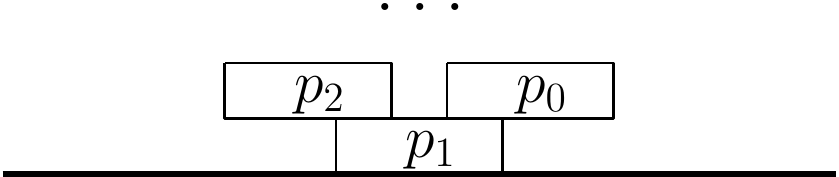}.
\end{center}
Let $H_2$ be the heap obtained from $H_1$ by a local flip at $(p_2,p_1,p_0)$.
Then $p_1$ is of rank 2 and $p_0,p_2$ are of rank 1 in $H_2$. Suppose that there
is a block $p_3$ in $H_2$ such that $p_3$ differs from $p_0$ and $(p_3,p_2,p_1)$
is flippable in $H_2$.  In this case, $p_3$ is necessarily of rank 2. Then let
$H_3$ be the heap obtained from $H_2$ by a local flip at $(p_3,p_2,p_1)$. $H_2$
and $H_3$ are depicted as follows:
\begin{center} \vspace{0.15cm}
\includegraphics[scale=0.7]{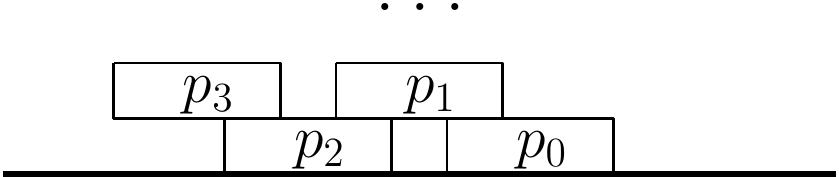} \quad
\includegraphics[scale=0.7]{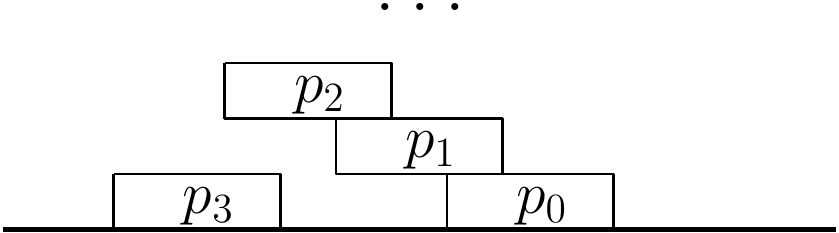}
\end{center}
Similarly, if there is a block $p_{i+1}$ in $H_i$ such that $p_{i+1}$ is different from $p_i$ and $(p_{i+1}, p_i, p_{i-1})$ is flippable, then $H_{i+1}$ is the heap obtained from $H_i$ by a local flip at $(p_{i+1}, p_i, p_{i-1})$. This process continues until there is no such $p_{i+1}$. Then we finally obtain the heap forming as follows:
\begin{center} \vspace{0.15cm}
\includegraphics[scale=0.75]{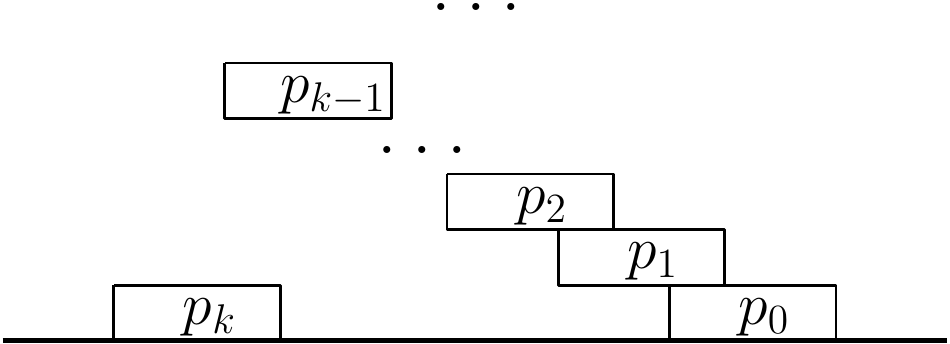}
\end{center}
By construction, $\{p_k, p_{k-1}, \dots, p_0\}$ forms a $P$-path. Also, deleting $p_k$ from this heap, the resulting heap has the unique sink $p_0$.
This process motivates the following definition. For a heap $H$, a set $\{v_{a_1, i_1}, \dots, v_{a_k,i_k}\}$ of blocks \emph{forms a forbidden $P$-path} if
\begin{enumerate}[label=(\roman*)]
\item $\{a_1,\dots,a_k\}$ forms a $P$-path,
\item $\rank(v_{a_1,i_1})=1$ and $\rank(v_{a_j,i_j})=k-j+1$ for $2\le j\le k$, and
\item $(v_{a_1,i_1}, v_{a_2,i_2}, v_{a_3,i_3})$ is flippable.
\end{enumerate}
Figure~\ref{fig:forbidden_path} illustrates an example and non-examples of a forbidden $P$-path.
\begin{figure}
\centering
\includegraphics[width=0.85\linewidth]{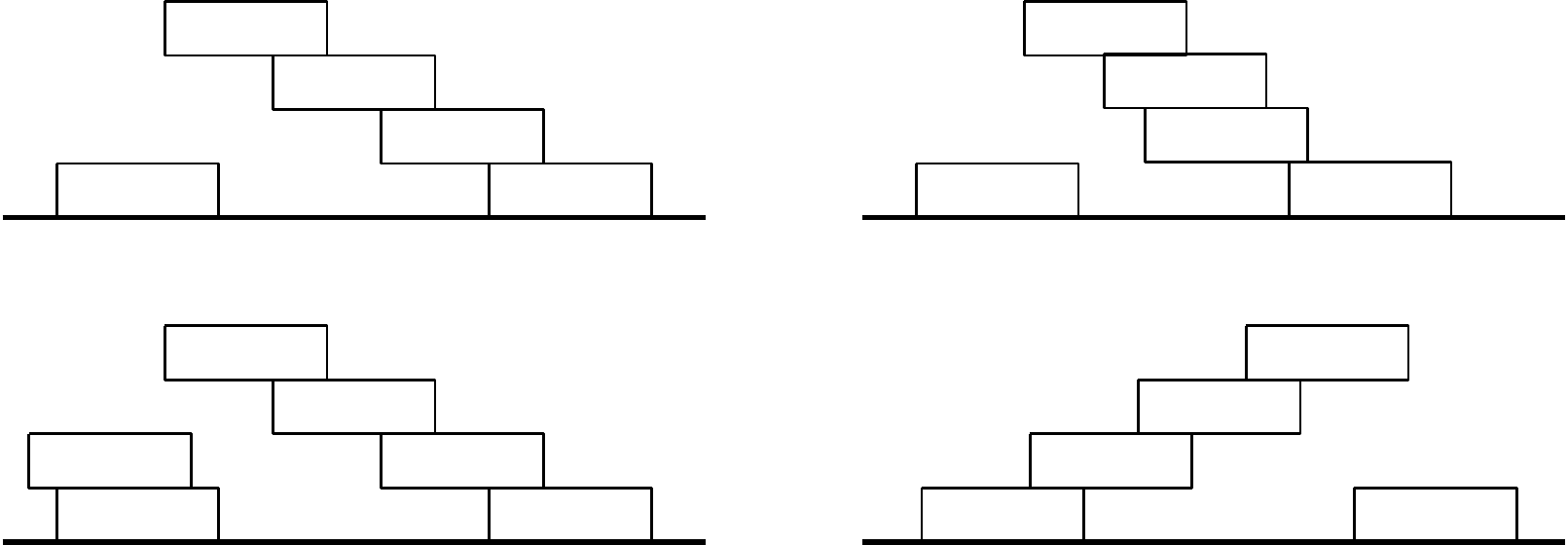}
\caption{
  The North-West heap forms a forbidden $P$-path.
  The others are non-examples of a forbidden $P$-path.
  The North-East one does not form a $P$-path. If we choose all blocks except
  the one of rank 3, we obtain a $P$-path, but they does not form a forbidden
  $P$-path because they violate the rank condition (ii). Similarly,
  the South-West diagram does not form a $P$-path, and if we choose all blocks
  except the one of rank 2 on left side, we have a $P$-path. But they violate
  the condition (iii), so is not a forbidden $P$-path. The last one is a $P$-path,
  but the blocks does not satisfy the rank condition (ii). Hence it does not
  form a forbidden $P$-path.} \label{fig:forbidden_path}
\end{figure}
The following lemma will be used in the proof of Theorem~\ref{thm:m_hook}.
\begin{lem} \label{lem:forbidden_path}
\begin{enumerate}[label=(\alph*),font=\upshape]
\item Every forbidden $P$-path can be transformed into a heap with a unique sink and two blocks of rank 2, via local flips.
\item Let $H$ be a heap of $P$ and $p$ a block in $H$. If $H$ contain a forbidden
$P$-path, and every forbidden $P$-path in $H$ contains $p$ as the rightmost element
in the diagram, then we can transform $H$ to a heap with no forbidden $P$-path via
local flips.
\end{enumerate}
\end{lem}
\begin{proof}
We replace a specific proof by diagrammatic examples.

(a) It follows from the definition of a $P$-path and the rank condition for a forbidden $P$-path.
\begin{center}
\includegraphics[width=0.8\linewidth]{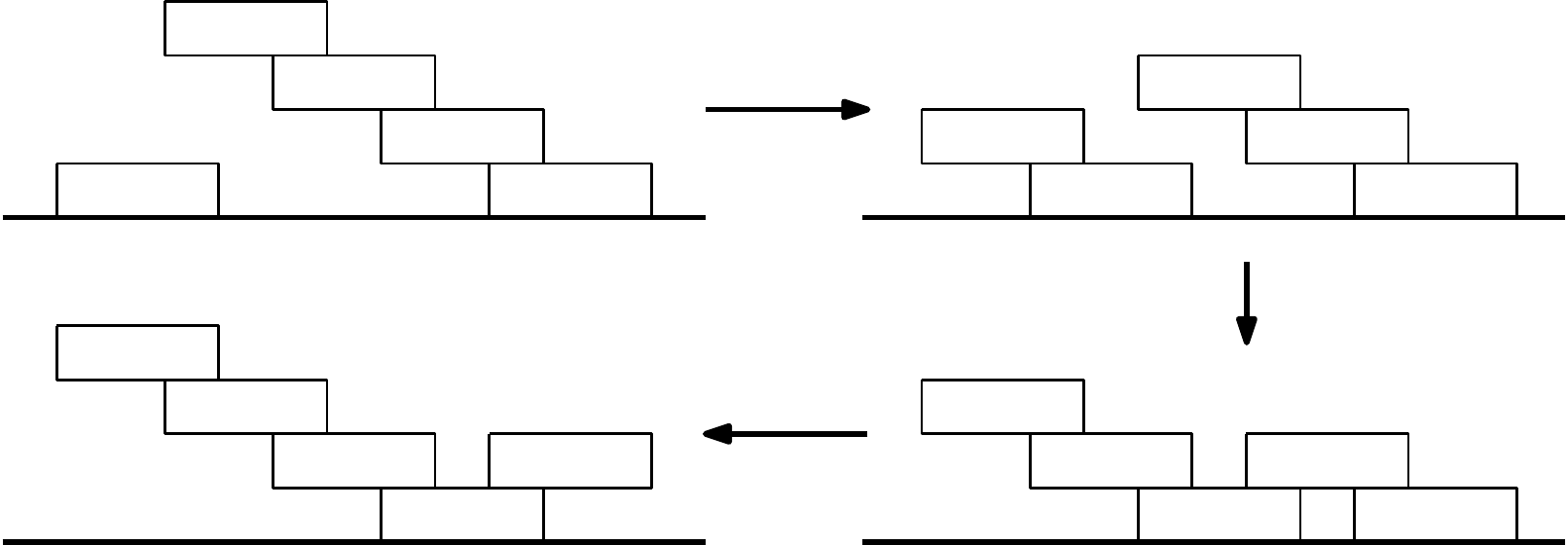}
\end{center}

(b) Choose the shortest forbidden $P$-path, and apply the procedure in (a) to it.
\begin{center}
\includegraphics[width=0.87\linewidth]{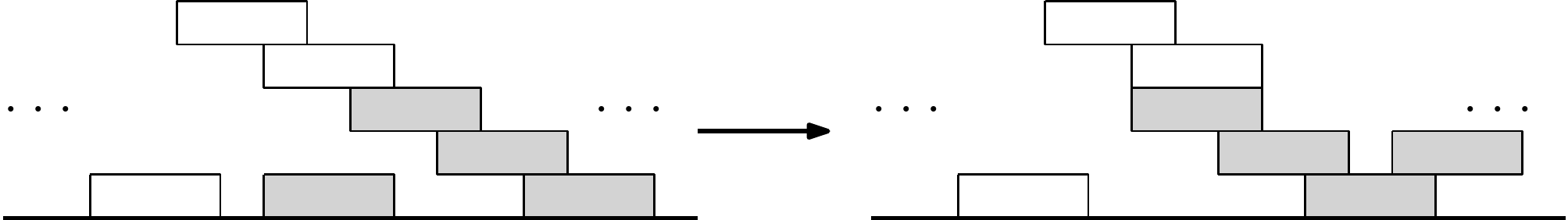}
\end{center}
\end{proof}

For $a,\ell\ge 1$, let $M_{(\ell|a)}$ be the set of heaps $H$ of $P$ consisting of $a+\ell+1$ such that $H$ has $\ell+1$ sinks and satisfies one of the following conditions:
\begin{enumerate}[label=(\roman*)]
\item the number of blocks of rank 2 is equal to 1;
\item the number of blocks of rank 2 is equal to 2. Also denoting the two blocks of rank 2 by $p$ and $r$,
\begin{enumerate}[label=(\Alph*)]
\item there is a block $q$ of rank 1 such that $(p,q,r)$ is a flippable triple,
\item there is no block $q'$ other than $q$ such that $p\rightarrow q'$ or $r\rightarrow q'$, and
\item $H$ has no forbidden $P$-path.
\end{enumerate}
\end{enumerate}
\begin{exam}
Let $a = 4$ and $\ell=2$. The following two heaps are examples of elements in $M_{(\ell\mid a)}$.
\begin{center}
\includegraphics[width=0.85\linewidth]{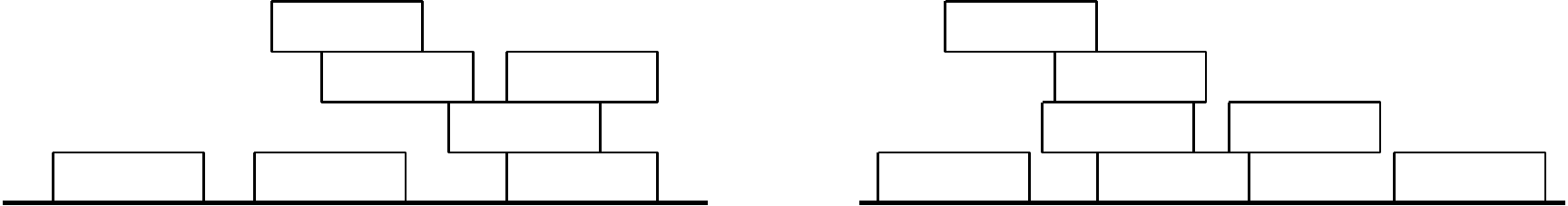}
\end{center}

In contrary, the following heaps are not included in $M_{(\ell\mid a)}$.
\begin{center}
\includegraphics[width=0.85\linewidth]{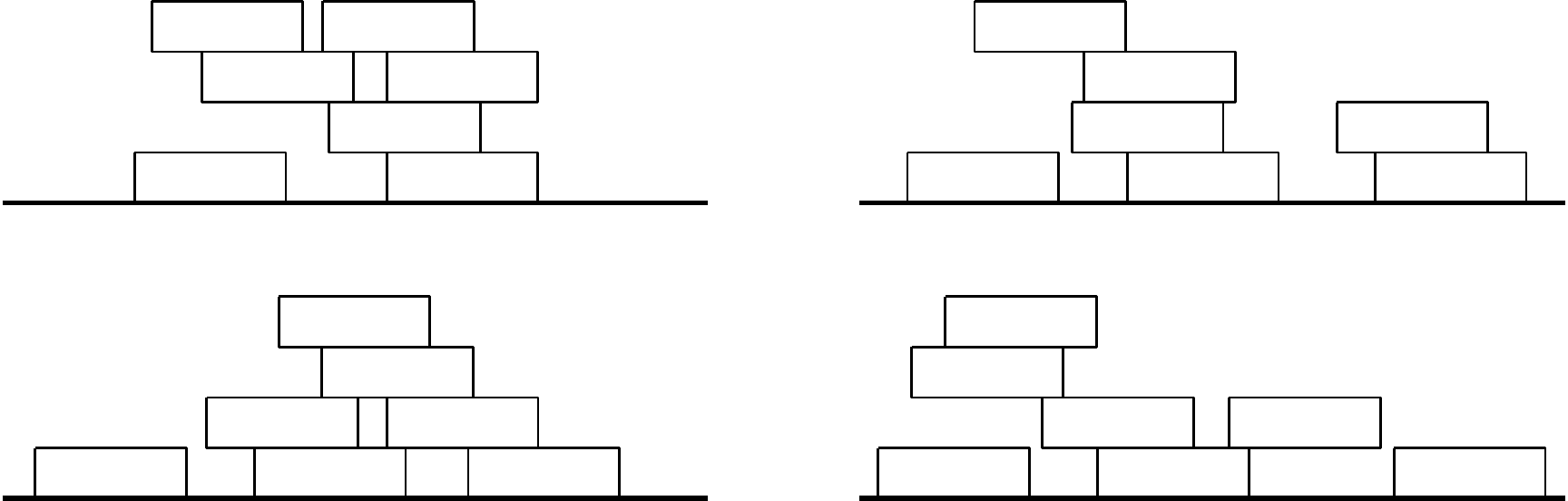}
\end{center}
Indeed, the North-West heap has only 2 sinks. The North-East, South-West, South-East heaps violate the conditions (A), (B), (C), respectively.
\end{exam}

As before, we first check the following identity related to the monomial symmetric
functions indexed by partitions of hook shapes: for $a\ge 2$ and $\ell\ge 0$, 
\begin{equation} \label{eq:ep=m_hook}
e_{\ell+1}(\xx) p_a(\xx) = m_{(a+1, 1^{\ell})}(\xx) + m_{(a, 1^{\ell+1})}(\xx).
\end{equation}
\begin{thm} \label{thm:m_hook}
For $a,\ell\ge 1$, we have
\begin{equation} \label{eq:mono_hook}
\mono_{(a+1,1^\ell)}(\uu) = \sum_{H\in M_{(\ell\mid a)}} u_{\ww_H} \mod \II_P.
\end{equation}
Consequently, let $X_P(\xx,q;\mu) = \sum_\lambda c_\lambda(q) e_{\lambda}(\xx)$, then
\[
c_{(a+1,1^\ell)}(q) = \sum_H q^{\asc_P(H)},
\]
where the sum runs over all heaps $H$ of type $\mu$ contained in $M_{(\ell\mid a)}$.
\end{thm}
\begin{proof}
We induct on $a$. For the base case $a=1$, the partition $(a+1,1^\ell)$ is not only of hook shape, but also of two-column shape. In this case, \eqref{eq:mono_two-column} and \eqref{eq:mono_hook} coincide.

For $a\ge 2$, from \eqref{eq:ep=m_hook}, we have
\begin{equation}
\ee_{\ell+1}(\uu) \pp_a(\uu) \equiv \mono_{(a+1,1^{\ell})}(\uu) + \mono_{(a,1^{\ell+1})}(\uu) \mod \II_P.
\end{equation}
Similar to the proof of Theorem~\ref{thm:m_two_column}, let $\EP_{\ell+1,a}$ be a multiset given by
\[
\EP_{\ell+1,a} = \left\{ \ww_1\cdots\ww_{\ell+a+1} ~\middle\vert \begin{array}{l}
    \ww_1 >_P \dots >_P \ww_{\ell+1} \mbox{ and $\ww_{\ell+2}\cdots\ww_{\ell+a+1}$ has} \\
    \mbox{no $P$-descents and no left-to-right $P$-maxima}
  \end{array}\right\},
\]
then
\[
\ee_{\ell+1}(\uu) \pp_a(\uu) \equiv \sum_{\ww\in\EP_{\ell+1,a}} u_{\ww} \mod \II_P.
\]
Hence by induction, it suffices to find a bijective map
\[
\Phi : \EP_{\ell+1,a} \longrightarrow M_{(\ell\mid a)} \cup M_{(\ell+1\mid a-1)}
\]
such that $H_\ww \sim \Phi(\ww)$ for all $\ww\in\EP_{\ell+1,a}$.

Let
\begin{align*}
M_{(\ell\mid a),1} &= \{ H\in M_{(\ell\mid a)} \mid \mbox{$H$ has a unique block of rank 2} \} \mbox{ and} \\
M_{(\ell\mid a),2} &= \{ H\in M_{(\ell\mid a)} \mid \mbox{$H$ has two blocks of rank 2} \},
\end{align*}
and let $M_{(\ell+1\mid a-1),1}, M_{(\ell+1\mid a-1),2}$ be defined similarly.
Recall that $\pp_k(\uu)$ is the generating function for heaps consisting of $k$ blocks with a unique sink; see Corollary~\ref{cor:pp=unique_sink}. For a word $\ww\in\EP_{\ell+1,a}$, blocks corresponding to $\ww_1,\dots,\ww_{\ell+1}$ form sinks in $H_\ww$ by definition, and the block corresponding to $\ww_{\ell+2}$ may or may not be a sink. Then $H_\ww$ has $\ell+1$ or $\ell+2$ sinks. With this fact and the diagram of $H_\ww$ in mind, we give an explicit case-by-case description of $\Phi(\ww)$ as follows:
\begin{description}
\item[Case 1] $H_\ww$ has $\ell+1$ sinks. In this case, there is a unique block of rank 2, namely the block corresponding to $\ww_{\ell+2}$. Then define $\Phi(\ww) = H_\ww \in M_{(\ell\mid a),1}$.
\item[Case 2] $H_\ww$ has $\ell+2$ sinks. Then $H_\ww$ has at most 2 blocks of rank 2. We break this case into subcases depending on the number of blocks of rank 2.
\begin{description}
\item[Case 2a] $H_\ww$ has a unique block of rank 2. We denote this block by $q$. There are 1 or 2 blocks under $q$ in the diagram of $H_\ww$.
\begin{description}
\item[Case 2aa] There is a unique block under $q$. Then this block has to correspond to $\ww_{\ell+2}$. In this case, we define $\Phi(\ww) = H_\ww \in M_{(\ell+1\mid a-1), 1}$.
\item[Case 2ab] There are two blocks under $q$. The corresponding diagram then forms as follows:
\begin{center} \vspace{0.15cm}
\includegraphics[scale=0.75]{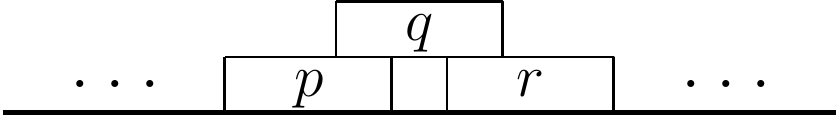}
\end{center}
Then $\ww_{\ell+2}$ corresponds to either $p$ or $r$. If it corresponds to $p$, then define $\Phi(\ww) = H_\ww\in M_{(\ell+1\mid a-1), 1}$. Otherwise, let $H'$ be the heap obtained from $H_\ww$ by a local flip at $(p,q,r)$. It is routine to check that $H'$ satisfies conditions (A), (B), (C), and hence define $\Phi(\ww) = H'\in M_{(\ell\mid a),2}$.
\end{description}
\item[Case 2b] $H_\ww$ has two blocks of rank 2. Then $H_\ww$ forms as
\begin{center} \vspace{0.15cm}
\includegraphics[scale=0.75]{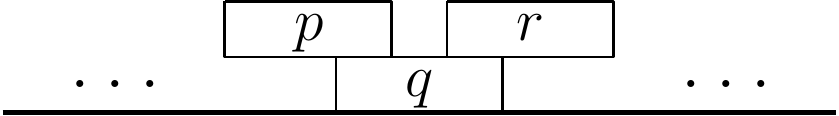},
\end{center}
where $q$ corresponds to $\ww_{\ell+2}$, and then satisfies condition (A). But it is possible that there is another block under $p$ or $r$. Again we split this case into subcases.
\begin{description}
\item[Case 2ba] $H_\ww$ satisfies conditions (B) and (C). Then define $\Phi(\ww) = H_\ww \in M_{(\ell+1\mid a-1),2}$.
\item[Case 2bb] There is a forbidden $P$-path containing $q$ as the rightmost
block. So $H_\ww$ violates condition (C), then we need to transform $H_\ww$.
Choose the shortest forbidden $P$-paths containing $q$ as the rightmost block,
and apply local flips as Lemma~\ref{lem:forbidden_path}(a).
Then Lemma~\ref{lem:forbidden_path}(b) guarantees that the resulting heap $H'$ satisfies conditions (A), (B), (C), so define $\Phi(\ww) = H'\in M_{(\ell\mid a),2}$.
\item[Case 2bc] There is no forbidden $P$-path having $q$ as the rightmost block, and there is another block under $r$. Then $H_\ww$ forms as
\begin{center} \vspace{0.15cm}
\includegraphics[scale=0.75]{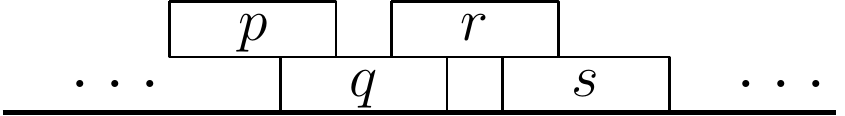}.
\end{center}
Let $H'$ be the heap obtained from $H_\ww$ by a local flip at $(q,r,s)$. Since $H_\ww$ has no forbidden $P$-path containing $q$ as the rightmost block, $H'$ do not allow any forbidden $P$-path. Moreover, conditions (A) and (C) hold for $H'$. We define $\Phi(\ww) = H'\in M_{(\ell\mid a),2}$.
\end{description}
\end{description}
\end{description}
By the construction of $\Phi$, we have that $H_\ww \sim \Phi(\ww)$ for all $\ww\in\EP_{\ell+1,a}$. Also it is straightforward that $\Phi$ is bijective, except the case where $\Phi(\ww)\in M_{(\ell\mid a), 2}$.
This happens in Cases 2ab, 2bb and 2bc. Let $\ww$ be a word such that $\Phi(\ww)\in M_{(\ell\mid a), 2}$, and $(p,q,r)$ be a unique flippable triple in $\Phi(\ww)$ such that $q$ is of rank 1 and $p,r$ are of rank 2.
Also let $H'$ be the heap obtained from $\Phi(\ww)$ by a local flip at $(p,q,r)$.
Observing $H'$ provides what case $\ww$ is in. First one can directly check that $\ww$ is in Case 2ab if and only if $H'$ has a unique block of rank $2$.
Suppose that $H'$ has at least two blocks of rank 2. If $\ww$ is in Case 2bb, then $H_\ww,H'$ and $\Phi(\ww)$ are depicted as follows:
\vspace{0.15cm}
\begin{align*}
H_\ww &= \vcenter{\hbox{\includegraphics[scale=0.75]{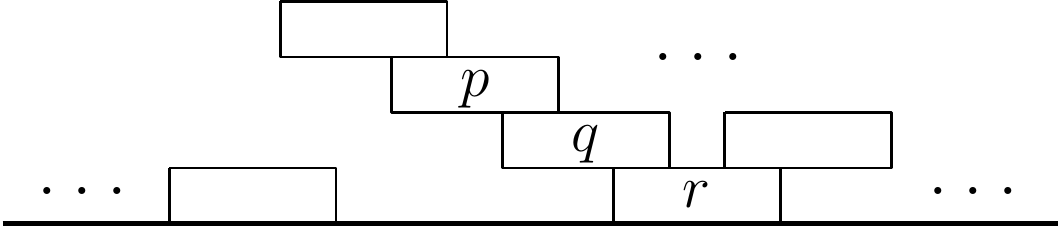}}} \\
\quad \\
H' &= \vcenter{\hbox{\includegraphics[scale=0.75]{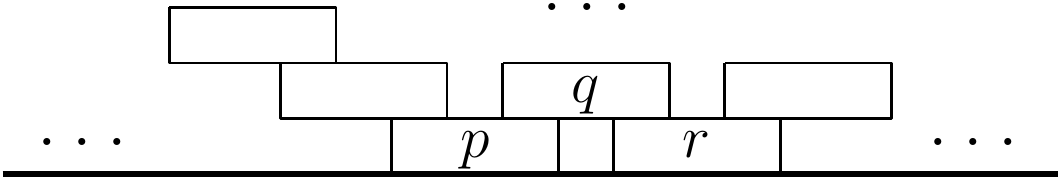}}} \\
\quad \\
\Phi(\ww) &= \vcenter{\hbox{\includegraphics[scale=0.75]{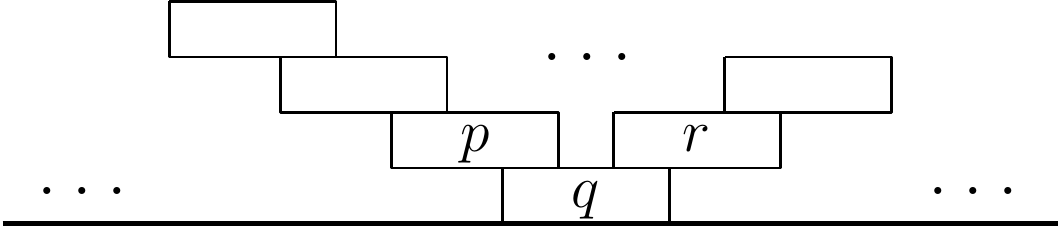}}}
\end{align*}
Hence, in $H'$, there is a block $s$ of rank 2 different from $q$ such that $s$ is placed above $r$.
On the other hand, if $\ww$ is in Case 2bc, then there is no such block $s$. Therefore we partition $M_{(\ell\mid a),2}$ as follows.
\begin{align*}
M'_{(\ell\mid a), 2} &= \{ H \in M_{(\ell\mid a), 2} \mid \mbox{$H'$ has a unique block of rank 2} \}, \\
M''_{(\ell\mid a), 2} &= \left\{ H \in M_{(\ell\mid a), 2} \middle\vert \begin{array}{l}
    \mbox{$H'$ has a block $s$ of rank 2} \\
    \mbox{different from $q$ such that $s\rightarrow r$}
  \end{array}\right\}, \mbox{ and} \\
M'''_{(\ell\mid a), 2} &= M_{(\ell\mid a), 2)} \setminus (M'_{(\ell\mid a), 2}\cup M''_{(\ell\mid a),2}).
\end{align*}
By the above observation, for a word $\ww$ such that $\Phi(\ww)\in M_{(\ell\mid a), 2}$, $\Phi(\ww)\in M'_{(\ell\mid a), 2}$ if $\ww$ is in Case 2ab, $\Phi(\ww)\in M''_{(\ell\mid a), 2}$ if $\ww$ is in Case 2bb, and $\Phi(\ww)\in M'''_{(\ell\mid a), 2}$ if $\ww$ is in Case 2bc. Using this partition, it is routine to show that the map
\[
\Phi: \EP_{\ell+1,a} \longrightarrow M_{(\ell\mid a),1}\cup M'_{(\ell\mid a),2}\cup M''_{(\ell\mid a),2}\cup M'''_{(\ell\mid a),2}\cup M_{(\ell+1\mid a-1),1}\cup M_{(\ell+1\mid a-1),2}
\]
is bijective.
\end{proof}

\begin{exam}
Let $P=P(2,3,4,5,5)$ and $\mu=(1^5)$ again. For $a=3$ and $\ell=1$, there are three heaps of type $\mu$ contained in $M_{(\ell\mid a)}$:
\begin{center}
\includegraphics[scale=1]{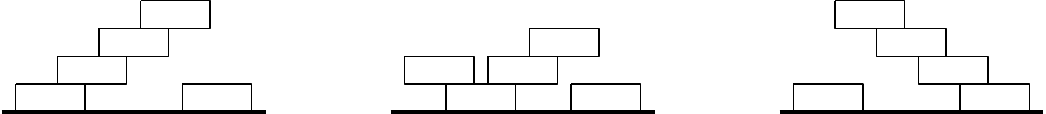}
\end{center}
From left to right, the ascent number of each heap is 3, 2, and 1, respectively. Then by Theorem~\ref{thm:m_hook}, we have $c_{4,1}(q) = q^3 + q^2 + q$. Also there is no heap of type $\mu$ contained in $M_{(2\mid 2)}$, and hence $c_{3,1,1}(q) = 0$.
\end{exam}

\subsubsection{A recurrence relation}
In this subsection, we give a recurrence relation for $\mono_\lambda(\uu)$,
which is equivalent to Harada--Precup's conjecture~\cite[Conjecture~8.1]{HP2019}.
Recall that for a natural unit interval order $P$, the height \( h \) of $P$ is
the maximum number of independent vertices in $\inc(P)$, or equivalently,
the maximum length of chains in $P$.
Then, by definition, $\ee_k(\uu) = 0$ for all $k>h$.

To show Harada--Precup's conjecture, we need the following simple lemma which
is straightforward from the definition of monomial symmetric functions.
\begin{lem} \label{lem:em=m}
For a partition $\nu$ and $k\ge \ell(\nu)$, let $e_k(\xx) m_\nu(\xx) = \sum_\rho d_\rho m_\rho(\xx)$. Then $d_{\nu+(1^k)} = 1$ where $\nu+(1^k) = (\nu_1+1,\nu_2+1,\dots,\nu_k+1)$, and if $\rho \neq \nu+(1^k)$ and $d_\rho \neq 0$, then $\ell(\rho) > k$.
\end{lem}

\begin{thm}[{\cite[Conjecture~8.1]{HP2019}}] \label{thm:m=em}
  Let \( P \) be a natural unit interval order, and $h$ the height of $P$.
  Then, for a partition $\lambda$ of length $\ell\ge h$, we have 
  \[
    \mono_\lambda(\uu) \equiv
    \begin{cases}
      0 & \mbox{if \( \ell > h \),} \\
      \ee_h(\uu)\mono_{\lambda^-}(\uu) & \mbox{if \( \ell = h \),}
    \end{cases}
  \]
  modulo $\II_P$, where ${\lambda^-} = (\lambda_1-1,\dots,\lambda_\ell-1)$.
\end{thm}
\begin{proof}
Since $\ee_k(\uu) = 0$ for all $k>h$, we have that $\ee_\nu(\uu) = 0$
for any partition $\nu = (\nu_1, \nu_2,\dots)$ with $\nu_1>h$. On the other hand,
by Proposition~\ref{prop:m=Ne}, if $\nu_1 < \ell$, then $N_{\lambda,\nu} = 0$.
Therefore we deduce that $\mono_\lambda(\uu) = 0$ if $\ell > h$.
Suppose now \( \ell=h \). As before, the noncommutative $P$-analogue of
Lemma~\ref{lem:em=m} also holds modulo \( \II_P \).
Thus applying Lemma~\ref{lem:em=m} to the case $\nu={\lambda^-}$ and $k=h$, we have
\[
\ee_h(\uu)\mono_{\lambda^-}(\uu) \equiv \mono_\lambda(\uu) + \sum_{\ell(\rho)>h} d_\rho \mono_\rho(\uu) \mod \II_P.
\]
However we already know that $\mono_\rho(\uu) \equiv 0$ module \( \II_P \)
for any partition $\rho$ with $\ell(\rho)>h$, and thus this completes the proof.
\end{proof}
By the same argument as in \cite{HP2019}, Theorem~\ref{thm:m=em} implies $e$-positivity of $X_P(\xx,q;\mu)$ for natural unit interval orders $P$ of height $2$.
We remark that Cho and Huh~\cite{CH2019} also deduced this result via a different approach.
\begin{cor}
  Let $P$ be a natural unit interval order of height $2$ and
  $\lambda=(\lambda_1,\dots,\lambda_\ell)$ a partition. Then we have
  \[
  \mono_\lambda(\uu) \equiv
  \begin{cases}
  ~ \pp_{\lambda_1}(\uu) & \mbox{if } \ell=1, \\
  ~ \ee_{2^{\lambda_2}}(\uu) \pp_{\lambda_1-\lambda_2}(\uu) & \mbox{if } \ell=2, \\
  ~ 0 & \mbox{if } \ell>2,
  \end{cases}
  \]
  modulo $\II_P$. Consequently, let $X_P(\xx,q;\mu) = \sum_\lambda c_\lambda(q)
  e_\lambda(\xx)$, then $c_\lambda(q)=0$ for any $\lambda$ with $\ell(\lambda)> 2$,
  and
  \[
  c_\lambda(q) = \sum_\ww q^{\inv_P(\ww)},
  \]
  where $\lambda = (\lambda_1, \lambda_2)$ ($\lambda_2$ may equal $0$) and
  $\ww$ ranges over all words such that for $1\le i\le \lambda_2$, $\ww_{2i}>_P
  \ww_{2i-1}$, and the consecutive subword
  $\ww_{2\lambda_2+1}\cdots\ww_{\lambda_1+\lambda_2}$ has no $P$-descents
  and no nontrivial left-to-right $P$-maxima.
\end{cor}
\begin{proof}
This is a direct consequence of Theorem~\ref{thm:m=em} and the fact that $\mono_k(\uu)\equiv \pp_k(\uu)$ modulo $\II_P$.
\end{proof}

As Harada and Precup mentioned, this recurrence relation does not in general give
positive \( \uu \)-monomial expressions of all noncommutative $P$-monomial symmetric
functions; see \cite[Example 8.2]{HP2019}.

\section{Future directions} \label{sec:future_direction}
In this paper, we defined a local flip on proper colorings of the incomparability
graph of a natural unit interval order, and refined the Shareshian--Wachs'
refinement. In addition, local flips allow us to use the theory of noncommutative
symmetric functions to the chromatic quasisymmetric functions.
In this direction, there are some open questions.

\subsection{Unit arc graphs}
Unit arc graphs are generalizations of the incomparability graphs of natural unit
interval orders. They have diagram models which are similar to the interval models
in Section~\ref{subsec:NUIO}. They are also described via a generalization of
Dyck paths.
The chromatic quasisymmetric functions of unit arc graphs are conjecturally
Schur positive, moreover, \( e \)-positive.
For details, see \cite{AP2018}.

Considering the diagrams of unit arc graphs, one can extend results in
Section~\ref{sec:heap_local_flip} to unit arc graphs in a similar way.
In particular, we can define a local flip on proper colorings of unit arc graphs.
Meanwhile, unit arc graphs unfortunately have no poset structure, so we cannot
apply the method in Section~\ref{sec:Noncomm}.
If one settle a suitable definition of noncommutative symmetric functions
for unit arc graphs, then we expect that the method in Section~\ref{sec:Noncomm}
can be applied to this case, and that one can prove some positivity of
the chromatic quasisymmetric functions of unit arc graphs.

\subsection{Hessenberg varieties}
One of remarkable connections between chromatic quasisymmetric functions and
other areas is the connection with Hessenberg varieties.
As mentioned in Section~\ref{sec:intro}, the chromatic quasisymmetric function
of a natural unit interval order \( P \) is equal to the graded Frobenius
characteristic of the cohomology ring of the regular semisimple
Hessenberg variety corresponding to \( P \), up to the omega involution.
The ascent statistic on the chromatic quasisymmetric function plays a grading
on the cohomology. Since local flips refine the ascent statistic
(Proposition~\ref{prop:ascent_number}) and each equivalence class defines
a symmetric function (Theorem~\ref{thm:X_P=Omega}) which is Schur positive
(Theorem~\ref{thm:s-positive}) and conjecturally \( e \)-positive
(Conjecture~\ref{conj:refined_e-positivity}),
we wonder how local flips appear in the cohomology of the Hessenberg variety.

\subsection{LLT polynomials}
In \cite{CM2018}, the authors discovered a relation between chromatic quasisymmetric
functions and unicellular LLT polynomials. They gave a plethystic formula between
them, and due to this formula, to study unicellular LLT polynomials is very
closely related to studying the chromatic quasisymmetric functions of natural unit
interval orders.
Hence one can ask whether there is a combinatorial operation for LLT polynomials,
similar to local flips.
Alexandersson introduced a similar flip on orientations of the incomparability
graph of natural unit interval orders, and the flip preserves the ascent statistic
for the LLT polynomials~\cite{Alex}.
In a similar way, the flip defines an equivalence relation on monomials in
the LLT polynomials, and then each equivalence class defines
a symmetric polynomial. It is an interesting question whether each such symmetric
function is Schur positive.

\section*{Acknowledgments}
The author would like to thank Jang Soo Kim and Per Alexandersson for suggestions
and useful discussions.
The author would also like to appreciate the referee for numerous suggestions that
have improved the clarity and readability of this paper.

\bibliographystyle{alpha}
\newcommand{\etalchar}[1]{$^{#1}$}

\end{document}